\documentclass[9pt, a5paper]{article}

\usepackage[english]{babel}
\usepackage[latin1]{inputenc}

\usepackage{ellipsis}		
\usepackage{xspace}			
\usepackage{exscale,relsize}		
\usepackage{ulem}			
\usepackage[hyphens]{url}
\usepackage{paralist}		
\usepackage{color,fancyvrb}		
\usepackage[force,full,safe]{textcomp} 	
\usepackage{float} 
\usepackage{fancyhdr}         
\usepackage{setspace}         

\usepackage{lmodern}
\usepackage{mathrsfs}
\usepackage{lipsum}
\usepackage{titlesec}

\usepackage[fleqn]{amsmath}		
\usepackage{amsthm}			
\usepackage{amsfonts}
\usepackage{amssymb}
\usepackage{bm,amscd} 

\usepackage{calc}
\usepackage{makeidx}
\usepackage{picinpar,graphicx}
\usepackage{tikz}
\usetikzlibrary{matrix,arrows,decorations.pathmorphing,positioning}
\usepackage[a4paper]{geometry}
\usepackage{makeidx}

\usepackage[colorlinks=true, linkcolor=blue, citecolor=brown]{hyperref}		

\newcounter{cntr}

\theoremstyle{plain}\newtheorem{thm}[cntr]{Theorem}
\theoremstyle{definition}\newtheorem{defn}[cntr]{Definition}
\theoremstyle{plain}\newtheorem{propn}[cntr]{Proposition}
\theoremstyle{plain}\newtheorem{cor}[cntr]{Corollary}
\theoremstyle{plain}\newtheorem{lem}[cntr]{Lemma}       
\theoremstyle{plain}
\theoremstyle{definition}
\theoremstyle{remark}\newtheorem{rem}[cntr]{Remark}
\theoremstyle{remark}

\newcommand\rarrow{\xrightarrow{\hspace*{0.7cm}}}

\title{On some properties of the stable cohomology of sheaves on the projective plane}
\author{Sayanta Mandal}

\titleformat{\section}{\normalfont\filcenter}{\thesection.}{12pt}{}
\newcommand\F{\mathbb{F}_1}
\renewcommand\P{\mathbb{P}^{2}}
\newcommand\MP{\mathcal{M}_{\mathbb{P}^2, H}}
\newcommand\MF{\mathcal{M}_{\mathbb{F}_1, E+F}}
\newcommand\MFF{\mathcal{M}_{\mathbb{F}_1, F}}
\renewcommand\L{\mathbb{L}}
\newcommand\Smu{S^{\mu} (\gamma_1, \cdots, \gamma_l ; F, E+F)}
\newcommand\mb{\bar{m}}

\fancypagestyle{abcd}{
\fancyhf{}

\fancyfoot[L]{
}}

\begin{document}
\thispagestyle{abcd}

\begin{center}
\begin{Large}
On the stabilization of the Betti numbers of the moduli space of sheaves on $\P$\\
\end{Large}
$\qquad$\\
\begin{large}
Sayanta Mandal
\end{large}
\end{center}

\paragraph{\normalfont\textsc{Abstract.}}
{\linespread{0.1}\footnotesize
Let $r \geq 2$ be an integer, and let $a$ be an integer coprime to $r$. We show that if $c_2 \geq n + \left\lfloor \frac{r-1}{2r}a^2 + \frac{1}{2}(r^2 + 1) \right\rfloor$, 
then the $2n$th Betti number of the moduli space $M_{\P,H}(r,aH,c_2)$ stabilizes, where $H = c_1(\mathcal{O}_{\P}(1))$.
}
\section{\textsc{Introduction}}
Let $X$ be a smooth projective surface over an algebraically closed field $\mathbb{K}$, and let $H$ be an ample divisor on $X$. We denote the Chern character of a torsion-free coherent  sheaf on $X$ by $\gamma = (r,c,\Delta)$, where $r$ is the rank, $c$ is the first Chern class, and $\Delta = \frac{ch_1^2 - 2\cdot r\cdot ch_2}{2r^2}$ is the discriminant. We denote by  $M_{X,H}(\gamma)$, the moduli-space parameterizing slope-$H$-semistable sheaves with Chern character $\gamma$. These spaces were constructed by Gieseker \cite{gie} and Maruyama \cite{mar}, and play a central role in many areas of mathematics including algebraic geometry, topology, representation theory, etc. For example, they are used to study linear systems on curves and in the Donaldson theory of $4$-manifolds. 

A crucial step to understand the geometry of these moduli spaces is by scrutinizing the cohomology groups associated with them. Consequently, determining the Betti numbers of these spaces are of utmost importance. In this paper, we look at the special case when $X = \P$ and $H = c_1(\mathcal{O}_{\P}(1))$. We show that \\
{}\\
\textbf{Theorem} (Theorem \ref{theorem26})\textbf{.} \textit{Assume that the rank $r$ and the first Chern class $aH$ are coprime. If $c_2 \geq N + \left\lfloor \frac{r-1}{2r}a^2 + \frac{1}{2}(r^2 + 1) \right\rfloor $, then the $2N$th Betti numbers of the moduli space $M_{\P,H}(r,aH,c_2)$ stabilizes.}\\

The general philosophy of Donaldson, Gieseker and Li is that the geometry of the moduli space $M_{X,H}(\gamma)$ behaves better as $\Delta$ tends to infinity. O'Grady \cite{ogr} showed that $M_{X,H}(\gamma)$ is irreducible and generically smooth if $\Delta$ is sufficiently large. Li \cite{li} showed the stabilization of the first and the second Betti numbers of $M_{X,H}(\gamma)$ when the rank is two.  When the rank is one, the moduli space $M_{X,H}(1,c,\Delta)$ is isomorphic to $Pic^c(X) \times X^{[\Delta]}$, where $X^{[n]}$ denotes the Hilbert scheme of $n$ points in $X$. The Betti numbers of $X^{[n]}$ were computed by G{\"{o}}ttsche \cite{got90}. Using the K{\"{u}}nneth formula, Coskun and Woolf \cite{cos}[Proposition 3.3] showed that the Betti numbers of $M_{X,H}(1,c,\Delta)$ stabilizes as $\Delta $ tends to infinity. In general, we don't know much about the Betti numbers of $M_{X,H}(\gamma)$. 

Yoshioka \cite{yos95}, \cite{yos} and G{\"{o}}ttsche \cite{got96} computed the Betti and Hodge numbers of $M_{X,H}(\gamma)$ when $X$ is a ruled surface and the rank is two. Yoshioka \cite{yos95}, \cite{yos96a} observed the stabilization of the Betti numbers for rank two bundles on ruled surfaces. G{\"{o}}ttsche \cite{got99} extended his results to rank two bundles on rational surfaces with polarizations which are $K_X$-negative. The stabilization of the Betti numbers is known for smooth moduli space of sheaves on $K3$ surfaces. By works of Mukai \cite{muk}, Huybrechts \cite{huy03}, and Yoshioka \cite{yos99}, smooth moduli spaces of sheaves on a  $K3$ surface $X$ are deformations of the Hilbert scheme of points on $X$ of the same dimension. In particular, they are diffeomorphic to the Hilbert scheme of points, and hence, their Betti numbers stabilizes. Yoshioka \cite{yos01} obtained similar results for moduli spaces of sheaves on abelian surfaces. A smooth moduli space of sheaves $M_{X,H}(\gamma)$ on an abelian surface $X$ is deformation equivalent to the product of the dual abelian surface of $X$ and a Hilbert scheme of points on $X$. Consequently, the Betti numbers stabilizes.

In the special case when $X = \P$ and $H = c_1(\mathcal{O}_{\P}(1))$, Yoshioka \cite{yos94},\cite{yos} computed the Betti numbers of $M_{\P,H}(2, -H,\Delta)$ and showed that they stabilizes as $\Delta$ tends to infinity. Manschot \cite{man11},\cite{man} computed the Betti numbers of $M_{\P,H}(3, -H, \Delta)$, and later, building on the work of Mozgovoy \cite{moz}, produced a formula to determine the Betti numbers for any rank and computed them in case of rank four. By looking at the tables present in the papers of Yoshioka and Manschot, one would expect the Betti numbers to stabilize as $\Delta$ tends to infinity, for any given rank and first Chern class. Coskun and Woolf \cite{cos} showed that this is indeed the case, and furthermore, they determined the generating function for the stable Betti numbers. Our goal in this paper is to produce lower bounds for the Betti numbers to become stable, and since we know the generating function for the stable Betti numbers, we can determine the Betti numbers for a large collection of moduli spaces.




\paragraph{Organization of the paper.} In section \ref{section2}, we set-up the notation and  review some basic facts on slope-semistable sheaves and their moduli space. In section \ref{section3}, we look at the Betti numbers of the moduli space of rank one sheaves on $\P$. In section \ref{section4}, we determine lower bounds for vanishing of the coefficient of $\L^{-N}q^{\Delta}$ in the generating function $\tilde{G}_{r,\tilde{c}}(q)$ (see equation \ref{generatingfndefn}) defined over the ring $A^-$ (see equation \ref{ringdefn}). In section \ref{section5}, we determine lower bounds for vanishing of the coefficient of $\L^{-N}q^{\Delta}$ in the generating function $G_{r,c}(q)$ (see equation \ref{generatingfndefn}). In section \ref{section6}, we determine lower bounds for $c_2$ for the stabilization of the Betti numbers of the moduli space $M_{\P,H}(r,aH,c_2)$.

\paragraph{Acknowledgements.} I am extremely grateful to my advisor Prof. Izzet Coskun for invaluable mathematical discussions, correspondences, and several helpful suggestions.

\section{\textsc{Preliminaries}}\label{section2}
Let $X$ be a smooth projective surface over an algebraically closed field $\mathbb{K}$ of characteristic zero, and let $H$ be an ample divisor on $X$. Throughout this paper, we are going to assume that all sheaves are coherent and torsion free. Given a sheaf $\mathcal{F}$, we define the $H$\textit{-slope} of $\mathcal{F}$ as 
\[ \mu_H(\mathcal{F}) = \frac{ch_1(\mathcal{F})\cdot H}{ch_0(\mathcal{F}) \cdot H^2} \]
Additionally, we define the \textit{Chern character} of $\mathcal{F}$ as $\gamma = (r,c,\Delta)$ where $r$ is the rank, $c$ is the first Chern class, and $\Delta$ is the discriminant defined as 
\begin{equation}
\label{deltadefn}
\Delta(\mathcal{F}) = \frac{ch_1(\mathcal{F})^2 - 2 ch_0(\mathcal{F}) ch_2(\mathcal{F})}{2 ch_0(\mathcal{F})^2}
\end{equation}
We define a sheaf $\mathcal{F}$ to be $\mu_H$\textit{-semistable} if for every proper subsheaf $\mathcal{E}$, we have $\mu_H(\mathcal{E}) \leq \mu_H(\mathcal{F})$. Likewise, we define a sheaf $\mathcal{F}$ to be $\mu_H$\textit{-stable} if the inequality is strict. Every $\mu_H$-semistable sheaf has a Jordan H{\"{o}}lder filtration with the sub-quotients being $\mu_H$-stable \cite{huy}[Proposition 1.5.2]. We say two $\mu_H$-semistable sheaves are \textit{S-equivalent} if the corresponding direct sum of subquotients appearing in the Jordan H{\"{o}}lder filtration are isomorphic.

Given a Chern character $\gamma = (r,c,\Delta)$, we denote by $M_{X,H}(\gamma)$ the moduli space of S-equivalence classes of $\mu_H$-semistable sheaves with Chern character $\gamma$. We denote by $\mathcal{M}_{X,H}(\gamma)$ the moduli stack of $\mu_H$-semistable sheaves with Chern character $\gamma$. When $X$ is smooth projective surface and $H$ is ample divisor with $K_X \cdot H <0$, the moduli space $M_{X,H}(\gamma)$ is smooth at every stable sheaf $\mathcal{F}$ because $ext^2(\mathcal{F},\mathcal{F}) = hom(\mathcal{F},\mathcal{F} \otimes K_X) = 0$. Consequently, if all $\mu_H$-semistable sheaves with Chern character $\gamma$ are $\mu_H$-stable, then $M_{X,H}(\gamma)$ is a smooth projective variety of dimension $ext^1(\gamma,\gamma) = 1 - \chi(\gamma,\gamma)$.

Assume that $M_{X,H}(\gamma)$ is smooth, e.g. when $r \cdot H^2$ and $c \cdot H$ are coprime. To understand the Betti numbers of $M_{X,H}(\gamma)$, we look at the polynomial 
\[ P_{M_{X,H}(\gamma)}(t) = \sum_{i=0}^{2(1 - \chi(\gamma,\gamma))} b_i(M_{X,H}(\gamma)) t^i \]

In general, consider a collection of polynomials $P_d(t) = \sum_{i=0}^{s_d} a_{i,d}t^i$ indexed by integers $d \geq N$, for some integer $N$. We look at the corresponding collection of shifted polynomials $\tilde{P}_d(t) = \sum_{j=-s_d}^0 b_{j,d}t^j$, where $b_{j,d} = a_{j+s_d,d}$. 
\begin{defn}\label{definition1}
We say that the collection of polynomials $P_d(t)$ \textit{stabilize} if for each $j$ there exists an integer $d_0(j)$ such that for all $d \geq d_0(j)$ we have $b_{j,d} = b_{j,d+1}$. In this case, we define the \textit{stable limit} to be $\tilde{P}_\infty (t) = \sum_{j=-\infty}^0 \beta_j t^j$, where $\beta_j = b_{j,d}$ for any $d \geq d_0(j)$. 
\end{defn}
In our case, we fix $r$ and $c$ and  look at the collection of polynomials $P_{M_{X,H}(r,c,\Delta)}$ for $\Delta \geq 0$. If this collection of polynomials stabilize, we say that the Betti numbers of $M_{X,H}(r,c,\Delta)$ \textit{stabilize}.

Consider the generating function 
\begin{equation}\label{defnFtilde}
\tilde{F}(q,t) = \sum_{d=N}^\infty \tilde{P}_d(t) q^d 
\end{equation}
We have 
\begin{propn}[\cite{cos}, Proposition 3.1]\label{proposition1}
The polynomials $P_d(t)$ stabilize iff the coefficient of $t^i$ in $(1-q)\tilde{F}(q,t)$ is a Laurent polynomial in $q$. Moreover, if the polynomials stabilize, the stable limit is obtained by evaluating $(1-q)\tilde{F}(q,t)$ at $q=1$.
\end{propn}

The proof of Proposition \ref{proposition1} due to Coskun and Woolf \cite{cos}  essentially follows from the following Lemma.
\begin{lem}\label{lemma3}
For any $j\geq 0$, the coefficient of $t^{-j}q^d$ in $(1-q)\tilde{F}(q,t)$ is zero for $d \geq d_0(j)$ iff $b_{j,d} = b_{j,d+1}$ for all $d \geq d_0(j) -1$.
\end{lem}
\begin{proof}
Let us define $b_{j,d}=0$ for $j < -s_d$. It follows from equation \ref{defnFtilde} that \[ \tilde{F}(q,t) = \sum_{d \geq N, \, j \leq 0} b_{j,d}t^j q^d \] whence, 
\[ (1-q)\tilde{F}(q,t) = \sum_{d \geq N, \, j \leq 0} (b_{j,d} - b_{j,d-1})t^jq^d\]
\end{proof}

Additionally, let 
\[ F(q,t) = \sum_{d=N}^\infty P_d(t) q^d \]
and assume that the polynomials $P_d(t)$ satisfy Poincar{\'{e}} duality i.e. $t^{s_d}P_d(t^{-1}) = P_d(t)$ for $d \gg 0$, then we have 
\begin{cor}[\cite{cos}, Corollary 3.2]
The polynomials $P_d(t)$ stabilize iff the coefficient of $t^i$ in $(1-q)F(q,t)$ is a Laurent polynomial in $q$, and in this case, we get the generating function for the stable coefficients by evaluating $(1-q)F(q,t)$ at $q=1$.
\end{cor}

Let $K_0(var_k)$ denote the Grothendieck ring of varieties over the field $k$ of characteristic zero. The Poincar{\'{e}} polynomials for smooth varieties induces \cite{joy07} the \textit{virtual Poincar{\'{e}} polynomial} map 
\[ P(t) : K_0(var_k) \rarrow \mathbb{Z}[t] \]
Let $\L$ denote the class $[\mathbb{A}^1]$ in $K_0(var_k)$. Consider the ring $R = K_0(var_k)[\L^{-1}]$. We have a $\mathbb{Z}$-graded filtration $\mathfrak{F}$ on $R$, where for any given variety $Y$, we have 
\[ [Y]\L^{a} \in \mathfrak{F}^i \quad\text{ iff } \quad dim(Y) + a \leq -i \]
We define the ring $A^{-}$ to be the inverse limit 
\begin{equation}\label{ringdefn}
A^- := \varprojlim_{i \geq 0} R/(\mathfrak{F}^i \otimes_{\mathfrak{F}^0} R) 
\end{equation}
Our notion of dimension extends from $K_0(var_k)$ to $A^-$. Similarly, the virtual Poincar{\'{e}} polynomial extends to $R$ and $A^-$ where it takes values in $\mathbb{Z}[t,t^{-1}]$ and $\mathbb{Z}((t^{-1}))$ respectively. 
\begin{defn}
We say that a sequence of elements $a_i$ in $A^-$ for $i \geq 0$ \textit{stabilize} to $a$ iff the sequence $a_i \L^{-dim(a_i)}$ converges to $a$. 
\end{defn}

Given any smooth projective variety $Y$ of dimension $d$, it follows from Poincar{\'{e}} duality that we have 
\begin{equation}\label{poincaredualeqn}
P_{[Y]}(t) = t^{2d}P_{[Y]}(t^{-1}) = P_{[Y]\L^{-d}}(t^{-1}) 
\end{equation}
Therefore, we have 
\begin{lem}\label{lemma6} Given a collection of smooth projective varieties $[X_i]$ of dimension $d_i$, if they stabilize in $A^-$ then their respective  Poincar{\'{e}} polynomials also stabilize.
\end{lem} 
Moreover, we know 
\begin{propn}[\cite{cos}, Proposition 3.6]\label{proposition6}
A sequence of elements $a_i \in A^-$ for $i \geq 0$ converges to $a$ iff the generating function $(1-q) \sum_{i\geq 0} a_i q^i$ is convergent at $q=1$, and in this case, evaluating the generating function $(1-q)\sum_{i\geq 0} a_i q^i$ at $q=1$ yields $a$. 
\end{propn}
In particular, we see that 
\begin{rem}\label{remark8}
If for all $N \geq 0$, there exists $\Delta_0(N) > 0$ such that the coefficient of $\L^{-N}q^{\Delta}$ in $(1-q)\sum_{i\geq 0} [X_i]\L^{-d_i}q^i$ is zero, then for all $N \geq 0$ the coefficient of $\L^{-N}$ in $(1-q)\sum_{i\geq 0} [X_i]\L^{-d_i}q^i$ is a Laurent polynomial of $q$ of degree at most $\Delta_0(N)$. As a result, it follows from Proposition \ref{proposition6} that the generating function $(1-q)\sum_i [X_i]\L^{-d_i}q^i$ is convergent at $q=1$, whence Lemma \ref{lemma6} yields the Poincar{\'{e}} polynomials of $[X_i]$ also stabilize. Consequently, it follows from equation \ref{poincaredualeqn}, Lemma \ref{lemma3}, and definition \ref{definition1} that the $2N$th Betti number of $X_\Delta$ stabilize when $\Delta \geq \Delta_0(N)-1$.
\end{rem}

Let $\F \rarrow \P$ be blow-up of $\P$ at a point $p$. Let $E$ be the exceptional divisor and let $F$ be the fiber class. We are going to look at the moduli stacks $\MP(r,c,\Delta)$ and $\MF(r,\tilde{c}, \tilde{\Delta})$ where $\gamma = (r,c,\Delta)$ is Chern character on $\P$ and $\tilde{\gamma} = (r,\tilde{c}, \tilde{\Delta})$ is Chern character on $\F$. We define generating functions 
\begin{equation}\label{generatingfndefn}
\begin{aligned}
& \qquad \qquad G_{r,c} (q) = \sum_{\Delta \geq 0} [\MP(r,c,\Delta)] \L^{r^2(1 - 2 \Delta)} q^{r\Delta} \\
& \text{and }\\
& \qquad \qquad \tilde{G}_{r,\tilde{c}}(q) = \sum_{\tilde{\Delta} \geq 0} [\MF(r,\tilde{c}, \tilde{\Delta})] \L^{r^2(1 - 2 \tilde{\Delta})} q^{r \tilde{\Delta}} 
\end{aligned}
\end{equation}
Coskun and Woolf have shown that  
\begin{thm}[\cite{cos}, Theorem 5.4, Corollary 5.5]
The generating function $(1-q)G_{r,c}(q)$ converges at $q=1$ to $\prod_{i=1}^\infty \frac{1}{(1 - \L^{-i})^3}$. Similarly, the generating function $(1-q)G_{r,\tilde{c}}(q)$ converges at $q=1$ to $\prod_{i=1}^\infty \frac{1}{(1 - \L^{-i})^4}$.
\end{thm}

Our goal is to determine lower bounds for the stabilization of the Betti numbers for the moduli space $M_{\P,H}(r,c,\Delta)$ in the special case when $r$ and $c \cdot H$ are coprime. The way we do this is by relating the stabilization of the Betti numbers with the convergence of the generating function $(1-q)G_{r,c}(q)$ at $q = 1$. A key ingredient in this method is to relate the classes of the moduli stack and the moduli space in $A$, which was shown by Coskun and Woolf, where $A$ is the quotient of $A^-$ by relations $[P] = [X][PGL_n]$ whenever $P \rarrow X$ is an {\'{e}}tale $PGL_n$-torsor.
\begin{propn}[\cite{cos}, Proposition 7.3]\label{proposition10}
The moduli stack and moduli space of $\mu_H$-stable sheaves on $X$, denoted $\mathcal{M}_{X,H}^s(\gamma)$ and $M_{X,H}^s(\gamma)$ respectively, are related in $A$ as follows: 
\begin{equation} \label{stacktospace}
[M_{X,H}^s (\gamma)] = (\L -1) [\mathcal{M}_{X,H}^s (\gamma)] 
\end{equation}
\end{propn}
By our assumption, $r$ and $c \cdot H$ are coprime, a posteriori, all $\mu_H$-semistable sheaves are $\mu_H$-stable. As a consequence, we can use Proposition \ref{proposition10} to relate the moduli stack and the moduli space.

\section{\textsc{Estimating the Generating Functions when the rank is one}}\label{section3}
In this section, our goal is to analyze the generating functions $G_{1,c}(q)$ and $\tilde{G}_{1,\tilde{c}}(q)$. More precisely, we are going to show that when $\Delta > 2N$ the coefficient of $\L^{-N}q^{\Delta}$ in the generating functions $(1-q)G_{1,c}(q)$ and $(1-q)\tilde{G}_{1,\tilde{c}}(q)$ is zero. As a consequence, we are going to show that the $2N$th Betti number of $M_{\mathbb{P}^2, H} (1,c,c_2)$ stabilize when $c_2 \geq 2N$.

Recall that given a smooth projective surface $X$ with an ample divisor $H$ on $X$, the moduli space $M_{X,H}(1,c,c_2)$ is isomorphic to $Pic^c(X) \times X^{[c_2]}$, where $Pic^c(X)$ is the abelian variety of line bundles on $X$ with first Chern class $c$, and $X^{[n]}$ is the Hilbert scheme of $n$ points on $X$. The Betti numbers of $X^{[n]}$ were computed by G{\"{o}}ttsche \cite{got90}. Using the K{\"{u}}nneth formula, Coskun and Woolf \cite{cos}[Proposition 3.3] showed that the Betti numbers of $M_{X,H}(1,c,c_2)$ stabilize as $c_2 $ tends to infinity. In the special case when $X = \mathbb{P}^2$, the moduli space $M_{\mathbb{P}^2, H}(1,c,c_2)$ is isomorphic to ${\mathbb{P}^2}^{[c_2]}$. Ellingsrud and Stromme  \cite{ell}[Theorem 1.1, Corollary 1.3] computed the Betti numbers of ${\mathbb{P}^2}^{[c_2]}$ and showed that the $2N$th Betti number stabilize when $c_2 \geq 2N$. In this section, our goal is to re-derive this result in a flavor similar to the higher rank case.


We infer from equation \ref{generatingfndefn} that 
\[ G_{1,c}(q) = \sum_{\Delta \geq 0} [\MP (r,c, \Delta)] \L^{(1-2\Delta)}q^\Delta \]
and 
\[ \tilde{G}_{1,\tilde{c}}(q) = \sum_{\tilde{\Delta} \geq 0} [\MF (r,\tilde{c},\tilde{\Delta})] \L^{(1-2\tilde{\Delta})} q^{\tilde{\Delta}} \]
We have
\begin{propn}\label{clm1}
For $\Delta > 2N$, the coefficient of $\L^{-N}q^{\Delta}$ in $(1-q)G_{1,c}(q)$ is zero. Same for $(1-q)\tilde{G}_{r, \tilde{c}}(q)$.
\end{propn}
\begin{proof} We have the following equality of generating functions due to G$\ddot{o}$ttsche \cite{got01}[Example 4.9.1] 
\[ \sum_{\Delta = 0}^\infty [(\P)^{[\Delta]}] q^{\Delta} = \prod_{m=1}^\infty \frac{1}{(1 - \L^{m-1}q^m)(1 - \L^m q^m)(1 - \L^{m+1}q^m)} \]
Replacing $q$ with $\L^{-2}q$ in above equation, we get
\[ \sum_{\Delta = 0}^\infty [(\P)^{[\Delta]}] \L^{-2 \Delta} q^{\Delta} = \prod_{m=1}^\infty \frac{1}{(1 - \L^{-(m-1)}q^{m})(1 - \L^{-m}q^m) (1 - \L^{-(m+1)}q^m)} \]
Note that we have \[ [\MP (1, c , \Delta)] = (\L -1)^{-1} [ (\P)^{[\Delta]}] \]
Thus, we get 
\begin{align*}
(1-q)G_{1,c}(q) &= \frac{(1-q)\L}{(\L-1)} \sum_{\Delta = 0}^\infty [(\P)^{[\Delta]}] \L^{-2\Delta} q^{\Delta} \\
&= \frac{(1-q)}{(1 - \L^{-1})} \prod_{m=1}^\infty \frac{1}{(1 - \L^{-(m-1)}q^{m})(1 - \L^{-m}q^m) (1 - \L^{-(m+1)}q^m)} \\
&= \prod_{m_1 = 2}^\infty \frac{1}{(1 - \L^{-(m_1 -1)} q^{m_1})} \prod_{m_2 = 1}^\infty \frac{1}{(1 - \L^{-2 m_2}q^{m_2})} \prod_{m_3 = 0}^\infty \frac{1}{(1 - \L^{-(m_3 +1)}q^{m_3})} \\
&= \prod_{m_1 = 2}^\infty \left( \sum_{\alpha_1 = 0}^\infty \L^{-(m_1 -1)\alpha_1}q^{m_1 \alpha_1} \right) \times \prod_{m_2 = 1}^\infty \left( \sum_{\alpha_2 = 0}^\infty \L^{-m_2 \alpha_2}q^{m_2 \alpha_2} \right) \times \\
&\qquad \qquad \qquad \qquad \prod_{m_3 = 0}^\infty \left( \sum_{\alpha_3 = 0}^\infty \L^{-(m_3 + 1) \alpha_3}q^{m_3 \alpha_3} \right)
\end{align*}
Each non-zero term contributing to the coefficient of $\L^{-N}q^{\Delta}$ in $(1-q)G_{1,c}(q)$ arises from a pair of equations 
\begin{align*}
\Delta &= \sum_{j=1}^{\delta_1} m_1^{(j)}\alpha_1^{(j)} + \sum_{j=1}^{\delta_2} m_2^{(j)}\alpha_2^{(j)} + \sum_{j=1}^{\delta_3} m_3^{(j)} \alpha_3^{(j)} \\
-N &= \sum_{j=1}^{\delta_1} -(m_1^{(j)}-1) \alpha_1^{(j)} + \sum_{j=1}^{\delta_2} -m_2^{(j)} \alpha_2^{(j)} + \sum_{j=1}^{\delta_3} - (m_3^{(j)} + 1) \alpha_3^{(j)}
\end{align*}
where $\alpha_1^{(j)}, \alpha_2^{(j)}, \alpha_3^{(j)} \geq 0$ for all $j\geq 1$, and $m_1^{(j)} \geq 2$, $m_2^{(j)} \geq 1$, $m_3^{(j)} \geq 0$ for all $j \geq 1$. Therefore, we see that 
\[ \Delta - N = \sum_{j=1}^{\delta_1} \alpha_1^{(j)} - \sum_{j=1}^{\delta_3} \alpha_3^{(j)} \leq \sum_{j=1}^{\delta_1} (m_1^{(j)} -1) \alpha_1^{(j)} \leq N \]
Hence, for $\Delta > 2N $ the coefficient of $\L^{-N}q^{\Delta}$ in $(1-q)G_{1,c}(q)$ must be zero.

In a similar fashion as above, we use the following equality of generating functions due to G$\ddot{o}$ttsche \cite{got01}[Example 4.9.3]
\[ \sum_{ \tilde{\Delta} =0}^\infty [\F^{[\tilde{\Delta}]}] q^{\tilde{\Delta}} = \prod_{m=1}^\infty \frac{1}{(1 - \L^{m-1}q^m)(1 - \L^m q^m)^2 (1 - \L^{m+1}q^m)} \]
Replacing $q$ with $\L^{-2} q$ and using the fact $[\MF(1,\tilde{c},\tilde{\Delta})] = (\L -1)^{-1} [ \F^{[\tilde{\Delta}]} ]$, we obtain the following equation 
\begin{align*}
(1-q) \tilde{G}_{1, \tilde{c}}(q) 
&= \prod_{m_1 = 2}^\infty \left( \sum_{\alpha_1 = 0}^\infty \L^{-(m_1 -1)\alpha_1}q^{m_1 \alpha_1} \right) \times \prod_{m_2 = 1}^\infty \left( \sum_{\alpha_2 = 0}^\infty \L^{-m_2 \alpha_2}q^{m_2 \alpha_2} \right)^2 \times \\
&\qquad \qquad \qquad \qquad \prod_{m_3 = 0}^\infty \left( \sum_{\alpha_3 = 0}^\infty \L^{-(m_3 + 1) \alpha_3}q^{m_3 \alpha_3} \right) 
\end{align*}
Each non-zero term contributing to the coefficient of $\L^{-N}q^{\Delta}$ in $(1-q)G_{1, \tilde{c}}(q)$ arises from a pair of equations 
\begin{align*}
\Delta &= \sum_{j=1}^{\delta_1} m_1^{(j)}\alpha_1^{(j)} + \sum_{j=1}^{\delta_{2,1}} m_2^{(j,1)}\alpha_2^{(j,1)} + \sum_{j=1}^{\delta_{2,2}} m_2^{(j,2)}\alpha_2^{(j,2)} + \sum_{j=1}^{\delta_3} m_3^{(j)} \alpha_3^{(j)} \\
-N &= \sum_{j=1}^{\delta_1} -(m_1^{(j)}-1) \alpha_1^{(j)} + \sum_{j=1}^{\delta_{2,1}} - m_2^{(j,1)}\alpha_2^{(j,1)} + \sum_{j=1}^{\delta_{2,2}} - m_2^{(j,2)}\alpha_2^{(j,2)}  + \sum_{j=1}^{\delta_3} - (m_3^{(j)} + 1) \alpha_3^{(j)}
\end{align*}
where $\alpha_1^{(j)}, \alpha_2^{(j,1)}, \alpha_2^{(j,2)}, \alpha_3^{(j)} \geq 0$ for all $j\geq 1$, and $m_1^{(j)} \geq 2$, $m_2^{(j,1)}, m_2^{(j,2)} \geq 1$, $m_3^{(j)} \geq 0$ for all $j \geq 1$. Therefore, we see that 
\[ \Delta - N = \sum_{j=1}^{\delta_1} \alpha_1^{(j)} - \sum_{j=1}^{\delta_3} \alpha_3^{(j)} \leq \sum_{j=1}^{\delta_1} (m_1^{(j)} -1) \alpha_1^{(j)} \leq N \]
Hence, for $\Delta > 2N$ the coefficient of $\L^{-N}q^{\Delta}$ in $(1-q)\tilde{G}_{1,\tilde{c}}(q)$ must be zero.
\end{proof} 

As a consequence of above Proposition \ref{clm1}, we have the following: 
\begin{propn}
When $c_2 \geq 2N$, the $2N$-th Betti number of $M_{\mathbb{P}^2, H}(1,c,c_2)$ stabilize.
\end{propn}
\begin{proof}
Note that all $\mu_H$-semistable sheaves of rank one on $\mathbb{P}^2$ are $\mu_H$-stable, because the rank is coprime to the first Chern class. As a consequence, we can use Proposition 
\ref{proposition10} due to Coskun and Woolf and the fact that $c_2 = r \Delta + \frac{r-1}{2r}c_1^2$ to get the following equality of generating functions 
\[ (1-q)\sum_{c_2 \geq 0} [M_{\mathbb{P}^2,H}(\gamma)]\L^{-ext^1(\gamma,\gamma)} q^{c_2} = (1-\L^{-1})(1-q)G_{1,c}(q) \]
where $\gamma$ denotes the Chern character $(r, c, \Delta)$.

Each term contributing to the coefficient of $\L^{-N}q^d$ in $(1-\L^{-1})(1-q)G_{1,c}(q)$ comes from a pair of equations 
\begin{align*}
d = \Delta
-N = \varepsilon - N'
\end{align*}
where $\varepsilon \in \{ -1,0 \}$ accounts for the contribution of the coefficient coming from $(1 - \L^{-1})$, and $(\Delta, N')$ accounts for the contribution coming from the terms in coefficient of $\L^{-N'}q^{\Delta}$ in $(1-q)G_{1,c}(q)$. It follows from Proposition \ref{clm1} that for the coefficient of $\L^{-N'}q^\Delta$ to be nonzero, we must have $\Delta \leq 2N'$. Consequently, we must have $d \leq 2N$. Hence, for $d>2N$, the coefficient of $\L^{-N}q^d$ in $(1-\L^{-1})(1-q)G_{1,c}(q)$ must be zero. Therefore, using Remark \ref{remark8}, we conclude that the $2N$th Betti number of $M_{\mathbb{P}^2, H}(1,c,c_2)$ stabilize for $c_2 \geq 2N$.
\end{proof}

\section{\textsc{Estimating the generating function $\tilde{G}_{r, \tilde{c}}(q)$ when rank is at least two}}\label{section4}
In this section, our goal is to show that there is a constant $C_0$ depending only on $r$ and $\tilde{c}$ such that when $\Delta > N + C_0$, the coefficient of $\L^{-N}q^{\Delta}$ in $(1-q)\tilde{G}_{r,\tilde{c}}(q)$ is zero. We are going to show this in a couple of steps. First, we are going to use Mozgovoy's theorem \cite{moz}[Theorem 1.1] and estimate a generating function in $A^-$ expressed in terms of the classes of the moduli stack $\MFF (\gamma)$. Then, we are going to use Joyce's theorem \cite{joy}[Theorem 6.21] to relate the classes of the moduli stacks $\MF (\gamma)$ and $\MFF (\gamma)$ in $A^-$. Lastly, we are going to use key ideas of Coskun and Woolf \cite{cos} and Manschot \cite{man11}, \cite{man} to derive our estimate (see Proposition \ref{boundMF}).

Throughout this section, we are going to assume that $r$ is at least two. We recall two theorems due to Mozgovoy \cite{moz} and Joyce \cite{joy} respectively. 

Let $\MFF (\gamma)$ denote the moduli stack of torsion free $\mu_F$ semistable sheaves on $\F$ with Chern character $\gamma = (r, c , \Delta)$. We define generating function 
\begin{equation}\label{eqnhrc}
H_{r,c}(q) = \sum_{\Delta \geq 0} [\MFF(r,c,\Delta)] q^{r\Delta} \end{equation}
Let $Z_{\mathbb{P}^1}(q) = \frac{1}{(1-q)(1-\L q)}$ be the motivic Zeta function for $\mathbb{P}^1$. Then, we have 
\begin{thm}[\cite{moz}[Theorem 1.1]\label{hrc} If $r \nmid c \cdot F$, then $\MFF (\gamma)$ is empty, and hence $H_{r,c}(q) = 0$. Otherwise, we have 
\[ H_{r,c}(q) = \frac{1}{(\L -1)} \prod_{i=1}^{r-1} Z_{\mathbb{P}^1}(\L^i) \prod_{k=1}^\infty \prod_{i=-r}^{r-1} Z_{\mathbb{P}^1}(\L^{rk+i}q^{k}) \]
\end{thm}

Before proceeding to Joyce's theorem, in a similar vein as in Proposition \ref{clm1}, we would like to show that for $\Delta \gg N$, the coefficient of $\L^{-N}q^{\Delta}$ in the generating function 

$(1 - q) \sum_{ \Delta \geq 0} [\MFF(r,c,\Delta)] \L^{r^2( 1 - 2 \Delta)} q^{r \Delta}$ vanishes.
\begin{propn}\label{hrc2}
If $\Delta > N $, the coefficient of $\L^{-N}q^{\Delta}$ in the generating function 
\[ (1 - q) \sum_{ \Delta \geq 0} [\MFF(r,c,\Delta)] \L^{r^2( 1 - 2 \Delta)} q^{r \Delta} \] is zero.
\end{propn}
\begin{proof} Clearly we can assume that $r \mid c \cdot F$, because otherwise by Mozgovoy's theorem (Theorem \autoref{hrc}) we have $[\MFF(r,c,\Delta)] = 0$.
Observe that 
\begin{equation}\label{eqnhrc2}
(1 - q) \sum_{ \Delta \geq 0} [\MFF(r,c,\Delta)] \L^{r^2( 1 - 2 \Delta)} q^{r \Delta} = (1-q) \L^{r^2} H_{r,c} (\L^{-2r}q) 
\end{equation}
Moreover, we have the following equations 
\begin{align*}
\frac{1}{(\L -1)} \prod_{i=1}^{r-1} \frac{1}{(1 - \L^i)(1 - \L^{i+1})} &= \frac{L^{-r^2}}{(1 - \L^{-r})} \prod_{i=1}^{r-1} \frac{1}{(1 - \L^{-i})^2} \\
\prod_{k=1}^{\infty} \prod_{i = -r}^{r-1} \frac{1}{(1 - \L^{-rk +i}q^k)(1 - \L^{-rk + i + 1}q^k)} &= \prod_{k_1 = 1}^\infty \frac{1}{(1 - \L^{-(rk_1 +r)}q^{k_1})} \times \\ 
& \prod_{k_2 = 1}^\infty \prod_{i=-r+1}^{r-1} \frac{1}{(1 - \L^{-(rk_2-i)}q^{k_2})^2} \times \\
& \prod_{k_3 =1}^\infty \frac{1}{(1 - \L^{-(rk_3 -r)}q^{k_3})}
\end{align*}
Therefore, we have 
\begin{align*}
(1-q) \L^{r^2} H_{r,c} (\L^{-2r}q) &= \left( \sum_{\alpha_1 = 0}^\infty \L^{-r \alpha_1} \right) \prod_{i=1}^{r-1} \left( \sum_{\alpha_2 = 0}^\infty \L^{-i \alpha_2} \right)^2 \prod_{k_1 = 1}^\infty \left( \sum_{\alpha_3 = 0}^\infty \L^{-(rk_1 + r) \alpha_3}q^{k_1 \alpha_3} \right) \times \\
& \prod_{k_2 = 1}^\infty \prod_{j = -r + 1}^{r-1} \left( \sum_{\alpha_4 = 0}^\infty \L^{-(rk_2 - j) \alpha_4} q^{k_2 \alpha_4} \right)^2 \prod_{k_3 = 2}^\infty \left( \sum_{\alpha_5 = 0}^\infty \L^{-(rk_3 - r)\alpha_5} q^{k_3 \alpha_5} \right)
\end{align*}
Each non-zero term contributing to the coefficient of $\L^{-N}q^{\Delta}$ in $(1 - q)\L^{r^2}H_{r,c}(\L^{-2r}q)$ corresponds to a pair of equations 
\begin{align*}
\Delta &= \sum_{j_1 = 1}^{\delta_1} k_1^{(j_1)} \alpha_3^{(j_1)} + \sum_{j_2 = 1}^{\delta_2} \sum_{j = -r + 1}^{r-1} k_2^{(j_2, j)} (\alpha_4^{(j_2,j,1)} + \alpha_4^{(j_2, j, 2)}) + \sum_{j_3 = 1}^{\delta_3} k_3^{(j_3)} \alpha_5^{(j_3)} \\
-N &= -r \alpha_1 + \sum_{i = 1}^{r-1} -i (\alpha_2^{(i,1)} + \alpha_2^{(i,2)}) + \sum_{j_1 = 1}^{\delta_1} -(rk_1^{(j_1)} +r) \alpha_3^{(j_1)} +  \\
& \qquad \qquad \sum_{j_2 = 1}^{\delta_2} \sum_{j = -r + 1}^{r-1} -(rk_2^{(j_2, j)} -j) (\alpha_4^{(j_2,j,1)} + \alpha_4^{(j_2, j, 2)}) + \sum_{j_3 = 1}^{\delta_3} -(rk_3^{(j_3)} -r) \alpha_5^{(j_3)}
\end{align*}
where all the $\alpha$'s are non-negative integers and all the $\delta$'s and $k$'s are positive integers except $k_3^{(j_3)}$ which is at least $2$, for all $1 \leq j_3 \leq \delta_3$. We see that 
\[ r\Delta -N \leq \sum_{j_2 =1}^{\delta_2} \sum_{j = -r+1}^{r-1} j(\alpha_4^{(j_2,j,1)} + \alpha_4^{(j_2, j, 2)}) + \sum_{j_3 = 1}^{\delta_3} r \alpha_5^{(j_3)} \]
Since $j \leq r-1$ and $k_3^{(j_3)} \geq 2$, we see that $ (rk_2^{(j_2,j)} -j) \geq 1$ and $ (rk_3^{(j_3)}-r) \geq r$. Hence, we have
{\small
\begin{align*}
\sum_{j_2 =1}^{\delta_2} \sum_{j = -r+1}^{r-1} j(\alpha_4^{(j_2,j,1)} + \alpha_4^{(j_2, j, 2)}) + \sum_{j_3 = 1}^{\delta_3} r \alpha_5^{(j_3)} &\leq (r-1)  \sum_{j_2 = 1}^{\delta_2} \sum_{j = -r + 1}^{r-1} (rk_2^{(j_2, j)} -j) (\alpha_4^{(j_2,j,1)} + \alpha_4^{(j_2, j, 2)}) \\
& \,\, +  \sum_{j_3 = 1}^{\delta_3} (rk_3^{(j_3)} -r) \alpha_5^{(j_3)} \leq (r-1)N 
\end{align*}}
Hence for $\Delta > N $, the coefficient of $\L^{-N}q^{\Delta}$ in 
$(1 - q) \sum_{ \Delta \geq 0} [\MFF(r,c,\Delta)] \L^{r^2( 1 - 2 \Delta)} q^{r \Delta}$ is zero.
\end{proof}

We now proceed to state Joyce's theorem. Let $X$ be a surface with two ample line-bundles $H_1$ and $H_2$. Let $\mathcal{M}_{X,H_1}(\gamma)$ (respectively $\mathcal{M}_{X,H_2}(\gamma)$) denote the moduli stack of torsion free $\mu_{H_1}$ (respectively $\mu_{H_2}$) semistable sheaves on $X$ with Chern character $\gamma = (r,c, \Delta)$. Let $\gamma_1, \cdots, \gamma_l$ be Chern characters such that $\sum_{i=1}^l \gamma_i = \gamma$. Assume that $l \geq 2$, and consider the following conditions for all $1 \leq i \leq l-1$
\begin{equation}\label{defnSmu}
\begin{aligned}
& \qquad \text{A) } \mu_{H_1}(\gamma_i) > \mu_{H_1}(\gamma_{i+1}) \text{ and } \mu_{H_2}(\sum_{j=1}^i \gamma_j) \leq \mu_{H_2} ( \sum_{j=i+1}^l \gamma_j)\\
& \qquad \text{B) } \mu_{H_1}(\gamma_i) \leq \mu_{H_1}(\gamma_{i+1}) \text{ and } \mu_{H_2}(\sum_{j=1}^i \gamma_j) > \mu_{H_2} ( \sum_{j=i+1}^l \gamma_j)
\end{aligned}
\end{equation}
Let $u$ be the number of times that Case B occurs. We define 
\begin{equation}\label{eqnSmu}
\begin{aligned}
 S^\mu( \gamma_1, \cdots, \gamma_l ;H_1,H_2) = \begin{cases} 1, &\text{ if } l=1 \\
 (-1)^u, &\text{ if } l\geq 2,  \text{ and Case A or B occurs for all }1 \leq i \leq l-1\\ 
 0, \qquad &\text{ otherwise }  \end{cases} 
 \end{aligned}
 \end{equation}
\begin{thm}[\cite{joy},{Theorem 6.21}]\label{joyce} If $H_1$ and $H_2$ are ample line-bundles on $X$ satisfying $K_X \cdot H_1 <0$ and $K_X \cdot H_2 < 0$, then we have the following equation 
\[ [ \mathcal{M}_{X,H_2}(\gamma)] = \sum_{\sum_{i=1}^l \gamma_i = \gamma} S^\mu(\gamma_1, \cdots, \gamma_l; H_1, H_2) \L^{- \sum_{1 \leq i < j \leq l} \chi(\gamma_j, \gamma_i)} \prod_{i=1}^l [\mathcal{M}_{X,H_1} (\gamma_i)] \]
\end{thm}
In our case, we would like to take $X = \F$, $H_1 = F$ and $H_2 = E + F$. Clearly, since $K_{\F} = -2 E - 3 F$, we have $K_{\F} \cdot H_1 < 0$ and $K_{\F} \cdot H_2 < 0$. However, $H_1$ is not ample and so we cannot use Joyce's theorem (Theorem \autoref{joyce}) as stated. Luckily the following observation due to Coskun and Woolf \cite{cos}[Corollary 4.4] saves the day. 
\begin{rem}
Joyce's theorem (Theorem \autoref{joyce}) holds if $H_1$ and $H_2$ are nef, as long as the sum on the right side of equation is convergent.
\end{rem}
Moreover, Coskun and Woolf shows \cite{cos}[Corollary 5.3] that we can use Joyce's equation in our case. Hence, we have 
\begin{equation}\label{eqn}
\begin{aligned}
\sum_{\Delta \geq 0} \MF(\gamma) q^{r \Delta} &= \sum_{\Delta \geq 0} \;\; \sum_{\sum_{i=1}^l \gamma_i = \gamma} \Smu \; \L^{-\sum_{1 \leq i< j \leq l} \chi(\gamma_j, \gamma_i)} \,\times \\
& \qquad \qquad \qquad \qquad \left( \prod_{i=1}^l [\MFF (\gamma_i)] \right) q^{r\Delta} 
\end{aligned}
\end{equation}

Let $\gamma_i = (r_i, c_i, \Delta_i)$ for all $1 \leq i \leq l$. Further, we define $\mu_i = \frac{c_i}{r_i}$ for all $1 \leq i \leq l$. We would like to manipulate equation \ref{eqn} so that the left hand side term of equation \ref{eqn} becomes $\tilde{G}_{r,c}(q)$ and get rid of $\Delta$ from the right hand side term of equation \ref{eqn}.

It is easy to see that 
\[ - \sum_{1 \leq i < j \leq l} \chi(\gamma_j, \gamma_i) = -\frac{1}{2} \left( \sum_{i<j} \chi(\gamma_j, \gamma_i) + \chi(\gamma_i, \gamma_j) \right) -\frac{1}{2} \left( \sum_{i<j} \chi(\gamma_j, \gamma_i) - \chi(\gamma_i, \gamma_j) \right) \]

We now list down some equations expressing the various Euler characteristics
\begin{enumerate}[$\qquad\bullet$]
\item $ \chi(\gamma_j, \gamma_i) - \chi(\gamma_i, \gamma_j) = r_i r_j (\mu_j - \mu_i) \cdot K_{\F} $
\item $\chi(\gamma, \gamma) = r^2( 1 - 2\Delta)$, and $\chi(\gamma_i, \gamma_i) = r_i^2 ( 1 - 2\Delta_i)$ for all $1 \leq i \leq l$.
\item $\sum_{i<j} \chi(\gamma_j, \gamma_i) + \chi(\gamma_i, \gamma_j) = \chi(\gamma, \gamma) - \sum_{i=1}^l \chi(\gamma_i, \gamma_i) $
\end{enumerate}
Using the above equations we get 
\begin{equation}\label{eqn2}
 - \sum_{i<j} \chi(\gamma_j, \gamma_i) = -\frac{1}{2}r^2(1 - 2\Delta) + \frac{1}{2}\sum_{i=1}^l r_i^2 (1 - 2 \Delta_i) - \frac{1}{2} \sum_{i<j} r_i r_j (\mu_j - \mu_i)\cdot K_{\F} 
\end{equation}
We now replace $q$ by $\L^{-2r}q$ in both sides of equation \ref{eqn}, multiply both sides of \ref{eqn} by $\L^{r^2}$, and use equation \ref{eqn2}. We get 
\begin{equation}\label{eqn3} 
\begin{aligned}
\sum_{\Delta \geq 0} [\MF (\gamma) ] \L^{r^2(1 - 2 \Delta)} q^{r\Delta} &= \sum_{\Delta \geq 0 } \;\; \sum_{\sum_{i=1}^l \gamma_i = \gamma} \Smu \;\times \\
& \L^{\frac{1}{2}r^2(1 - 2 \Delta) + \frac{1}{2}\sum_{i=1}^l r_i^2( 1 - 2 \Delta_i)} 
\;  \L^{- \frac{1}{2}\sum_{i<j}r_i r_j(\mu_j - \mu_i) \cdot K_{\F}} \;\times \\
& \left( \prod_{i=1}^l [\MFF (\gamma_i) ] \right) q^{r\Delta}  
\end{aligned}\end{equation}
Note that we are yet to get rid of $\Delta$ from right hand side term in equation \ref{eqn3}. To do that, we need to use Yoshioka's relation for discriminants \cite{yos}[Equation 2.1]
\begin{equation}\label{eqnyos}
r\Delta = \sum_{i=1}^l r_i \Delta_i - \sum_{i=2}^l \; \frac{1}{2r_i \left( \sum_{j=1}^i r_j \right) \left( \sum_{j=1}^{i-1}r_j \right)} \left( \sum_{j=1}^{i-1} r_i c_j - r_j c_i \right)^2
\end{equation}
It follows from Yoshioka's relation that the difference $r \Delta - \sum_{i=1}^l r_i \Delta_i$ depends only on $(r,c)$ and $(r_i,c_i)$ for $1 \leq i \leq l$. So we rewrite the first exponent of $\L$ in equation \ref{eqn3}
\begin{equation}\label{eqn5}
\frac{1}{2}r^2(1 - 2 \Delta) + \frac{1}{2}\sum_{i=1}^l r_i^2 (1 - 2 \Delta_i) = \frac{1}{2}(r^2 + \sum_{i=1}^l r_i^2) - r(r \Delta - \sum_{i=1}^l r_i \Delta_i) - \sum_{i=1}^l r_i (r + r_i) \Delta_i
\end{equation}
{}\\
Using equation \ref{eqn5} back in equation \ref{eqn3} yields
\begin{equation}\label{eqn6}
\begin{aligned}
\tilde{G}_{r,c}(q) &= \sum_{\Delta \geq 0} \;\; \sum_{\sum_{i=1}^l \gamma_i = \gamma } \Smu 
  \L^{\frac{1}{2}\left( r^2 + \sum_{i=1}^l r_i^2 \right)} \; 
  \L^{-\frac{1}{2}\sum_{i<j}r_i r_j(\mu_j - \mu_i) \cdot K_{\F}} \;\times \\
& \qquad \qquad \qquad \qquad
  \left( \L^{-r} q \right)^{r\Delta - \sum_{i=1}^l r_i \Delta_i}
  \left( \prod_{i=1}^l [\MFF (\gamma_i)] \left( \L^{-(r + r_i)} q \right)^{r_i \Delta_i} \right)
\end{aligned}
\end{equation}

Observe that all the terms except the last one involving products on right hand side of equality in equation \ref{eqn6} depends only on $(r,c)$ and $(r_i, c_i)$ for $1 \leq i \leq l$, and the last term depends only on the $\Delta_i$'s for $1 \leq i \leq l$. Therefore, we have 
\begin{equation}\label{eqn7}
\begin{aligned}
\tilde{G}_{r,c}(q) &= \sum_{\sum_{i=1}^l \gamma_i = \gamma} \Smu 
  \L^{\frac{1}{2}\left( r^2 + \sum_{i=1}^l r_i^2 \right)} \; 
  \L^{-\frac{1}{2}\sum_{i<j}r_i r_j(\mu_j - \mu_i) \cdot K_{\F}} \;\times \\
& \qquad \qquad \qquad 
  \left( \L^{-r} q \right)^{r\Delta - \sum_{i=1}^l r_i \Delta_i}
  \sum_{\Delta_1, \cdots, \Delta_l}
  \left( \prod_{i=1}^l [\MFF (\gamma_i)] \left( \L^{-(r + r_i)} q \right)^{r_i \Delta_i} \right)
\end{aligned}
\end{equation}
Recall that we previously defined in equation \ref{eqnhrc} the generating function 
\[ H_{r,c}(q) = \sum_{\Delta \geq 0} [\MFF(r,c,\Delta)] q^{r\Delta} \]
The second summation term in equation \ref{eqn7} can be expressed in terms of $H_{r,c}(q)$ as follows
\begin{equation}\label{eqn8}
  \sum_{\Delta_1, \cdots, \Delta_l}
  \left( \prod_{i=1}^l [\MFF (\gamma_i)] \left( \L^{-(r + r_i)} q \right)^{r_i \Delta_i} \right)
= \prod_{i=1}^l H_{r_i, c_i} (\L^{-(r+r_i)}q)
\end{equation}
Therefore, we have
\begin{equation}
\label{eqn9}
\begin{aligned}
\tilde{G}_{r,c}(q) &= \sum_{\sum_{i=1}^l (r_i,c_i) = (r,c)} \Smu 
  \L^{\frac{1}{2}\left( r^2 - \sum_{i=1}^l r_i^2 \right)} \; 
  \L^{-\frac{1}{2}\sum_{i<j}r_i r_j(\mu_j - \mu_i) \cdot K_{\F}} \;\times \\
& \qquad \qquad \qquad 
  \left( \L^{-r} q \right)^{r\Delta - \sum_{i=1}^l r_i \Delta_i}
  \prod_{i=1}^l \L^{r_i^2} H_{r_i, c_i} (\L^{-(r+r_i)}q)
\end{aligned}
\end{equation}

It follows from the definition of $\Smu$ in equation \ref{eqnSmu} and from Mozgovoy's theorem (Theorem \autoref{hrc}) that all the terms on right hand side of equality of equation \ref{eqn9} depends only on $(r,c)$ and $(r_i,c_i)$ for $1 \leq i \leq l$. Our next goal is to analyze the exponents of each of these terms further and show that for $\Delta \gg N$ the coefficient of $\L^{-N}q^{\Delta}$ in $(1-q)\tilde{G}_{r,c}(q)$ vanishes.

\begin{propn}\label{boundMF}
There is a constant $C_0$ depending only on $r$ and $c$ such that if $\Delta > N  + C_0$, then coefficient of $\L^{-N}q^{\Delta}$ in $(1-q)\tilde{G}_{r,c}(q)$ is zero. Moreover, we can take $C_0$ to be $\frac{1}{2}(r^2 + 1)$.
\end{propn}
\begin{proof}
Our approach is to look at each summand of $(1-q)\tilde{G}_{r,c}(q)$ corresponding to a equation \[ (r,c) = \sum_{i=1}^l (r_i,c_i) \]
and find a lower bound for $\Delta$ corresponding to the term 
\begin{equation}\label{rhs}
\begin{aligned} 
& (1-q)\Smu 
  \L^{\frac{1}{2}\left( r^2 - \sum_{i=1}^l r_i^2 \right)} \; 
  \L^{-\frac{1}{2}\sum_{i<j}r_i r_j(\mu_j - \mu_i) \cdot K_{\F}} \;\times \\
&   \qquad \qquad \qquad 
  \left( \L^{-r} q \right)^{r\Delta - \sum_{i=1}^l r_i \Delta_i}
  \prod_{i=1}^l \L^{r_i^2} H_{r_i, c_i} (\L^{-(r+r_i)}q)
\end{aligned} 
\end{equation}

If $l=1$, then equation \ref{rhs} becomes
\begin{equation}
(1-q) S^\mu(\gamma; F, E+F)\L^{r^2} H_{r,c}(\L^{-2r}q)
\end{equation}
It follows from Proposition \ref{hrc2} and equation \ref{eqnhrc2} that for $\Delta > N $, the coefficient of $\L^{-N}q^{\Delta}$ in $(1-q)\L^{r^2} H_{r,c}(\L^{-2r}q)$ is zero.

Assume $l \geq 2$. We would like to estimate a lower bound for $\Delta'_i$ such that the coefficient of $\L^{-N'_i}q^{\Delta'_i}$ in $\L^{r_i^2}H_{r_i,c_i}(\L^{-(r+r_i)}q)$ is zero, and then use that to figure out a lower bound for $\Delta$ in equation \ref{rhs}. It follows from Mozgovoy's theorem (Theorem \ref{hrc}) that 
\begin{align*}
\L^{r_i^2} H_{r_i,c_i}(\L^{-(r+r_i)}q) &= \L^{r_i^2} \frac{1}{(\L -1)} \prod_{j=1}^{r_i -1} Z_{\mathbb{P}^1}(\L^j) \prod_{k=1}^\infty \prod_{j=-r_i}^{r_i -1} Z_{\mathbb{P}^1} (\L^{-(rk -j)}q^k) \\
&= \frac{1}{(1-\L^{-r_i})} \left( \prod_{j=1}^{r_i -1} \frac{1}{(1 - \L^{-j})^2} \right) 
 \prod_{k=1}^\infty \Bigg\lbrace \frac{1}{(1 - \L^{-(rk + r_i)}q^k)} \;\times \\
& \qquad \qquad \qquad  \left( \prod_{j= -r_i +1}^{r_i -1} \frac{1}{(1 - \L^{-(rk - j)}q^k)^2} \right) \frac{1}{(1 - \L^{-(rk-r_i)}q^k)} \Bigg\rbrace
\end{align*}
Thus, we get 
\begin{align*}
\L^{r_i^2} H_{r_i,c_i}(\L^{-(r +r_i)}q) &= \left( \sum_{\alpha_1 =0}^\infty \L^{-r_i \alpha_1} \right) 
\left( \prod_{j_1 =1}^{r_i -1} \left( \sum_{\alpha_2 =0}^\infty \L^{-j_1 \alpha_2} \right)^2 \right) \: \times \\
&
\prod_{k=1}^\infty \Bigg\lbrace
\left( \sum_{\alpha_3 =0}^\infty \L^{-(rk + r_i) \alpha_3}q^{k \alpha_3} \right)
\left( \prod_{j_2 = -r_i +1}^{r_i -1} \left( \sum_{\alpha_4 =0}^\infty \L^{-(rk - j_2) \alpha_4} q^{k \alpha_4} \right)^2 \right) \\
& \qquad \qquad
\left( \sum_{\alpha_5 =0}^\infty \L^{-(rk - r_i) \alpha_5} q^{k \alpha_5} \right) \Bigg\rbrace
\end{align*}

Each nonzero term contributing to the coefficient of $\L^{-N'_i}q^{\Delta'_i}$ in $\L^{r_i^2}H_{r_i,c_i}(\L^{-(r+r_i)}q)$ arises from a pair of equations
\begin{align*}
\Delta'_i &= \sum_{j=1}^\delta \left\lbrace k^{(j)}\alpha_3^{(j)} + \sum_{j_2 = -r_i +1}^{r_i -1} k^{(j)}( \alpha_4^{(j,j_2,1)} + \alpha_4^{(j,j_2,2)}) + k^{(j)}\alpha_5^{(j)} \right\rbrace \\
-N'_i &= -r_i \alpha_1 + \sum_{j_1=1}^{r_i -1} -j_1 (\alpha_2^{(j_1,1)} + \alpha_2^{(j_1,2)}) + 
\sum_{j=1}^\delta \Bigg\lbrace -(rk^{(j)} +r_i) \alpha_3^{(j)} \; + \\
& \qquad
\left( \sum_{j_2 = -r_i +1}^{r_i -1} -(rk^{(j)} -j_2) (\alpha_4^{(j,j_2,1)} + \alpha_4^{(j,j_2,2)}) \right) 
-(rk^{(j)} -r_i)\alpha_5^{(j)} \Bigg\rbrace
\end{align*}
where all the $\alpha$'s are non-negative integers, $\delta$ and the $k$'s are positive  integers. Hence, we get 
\[ r \Delta'_i - N'_i \leq \sum_{j=1}^\delta \left( \sum_{j_2 = -r_i +1}^{r_i -1} j_2 (\alpha_4^{(j,j_2,1)}  + \alpha_4^{(j,j_2,2)} ) \right) + r_i \alpha_5^{(j)} \]
Since $j_2 \leq r_i -1$ and $k^{(j)} \geq 1$, we see that 
$ j_2 \leq r_i (rk^{(j)} - j_2) $. Moreover, because $l \geq 2$ we have $r_i \leq (r-1)$, and so
$ r_i \leq r_i (rk^{(j)} -r_i) $. These two inequalities yield 
\[ \sum_{j=1}^\delta \left( \sum_{j_2 = -r_i +1}^{r_i -1} j_2 (\alpha_4^{(j,j_2,1)}  + \alpha_4^{(j,j_2,2)} ) \right) + r_i \alpha_5^{(j)} \leq r_i N'_i \]
In summary, we get $r \Delta'_i - N'_i \leq r_i N'_i \leq (r-1) N'_i$, a posteriori, $\Delta'_i \leq N'_i$.

Going back to equation \ref{rhs}, we see that each non-zero term contributing to the coefficient of $\L^{-N'}q^{\Delta'}$ in equation \ref{rhs} arises from a pair of equations 
\begin{align*}
\Delta' &= \varepsilon + \left( r\Delta - \sum_{i=1}^l r_i \Delta_i \right) + \sum_{i=1}^l \Delta'_i \\
-N' &= \frac{1}{2}\left( r^2 - \sum_{i=1}^l r_i^2 \right) - \frac{1}{2} \left( \sum_{i<j} r_i r_j (\mu_j - \mu_i) \cdot K_{\F} \right) - r \left( r\Delta - \sum_{i=1}^l r_i \Delta_i \right)  + \sum_{i=1}^l - N'_i
\end{align*}
where $\varepsilon \in \{ 0,1 \}$ which accounts for contribution to the coefficient coming from $(1-q)$, and $(\Delta'_i, N'_i)$ accounts for the contribution of terms to the coefficient of $\L^{-N'}q^{\Delta'}$ coming from terms of coefficient of $\L^{-N'_i}q^{\Delta'_i}$ appearing in $\L^{r_i^2}H_{r_i,c_i}(\L^{-(r+r_i)}q)$. Since $\Delta'_i \leq N'_i$ for all $1 \leq i \leq l$ and $\varepsilon \leq 1$, we see that 
\begin{equation}\label{eqn15}
\Delta' \leq N' + 1 + \frac{1}{2}\left( r^2 - \sum_{i=1}^l r_i^2 \right) - \frac{1}{2} \left\lbrace \left( \sum_{i<j} r_i r_j (\mu_j - \mu_i) \cdot K_{\F} \right) + 2(r-1)\left( r\Delta - \sum_{i=1}^l r_i \Delta_i \right) \right\rbrace 
\end{equation}
Clearly, to bound $\Delta'$, we need to bound the last term in above equation \ref{eqn15}. We are going to show later (in Lemma \autoref{lowerbound}) that 
\[ 2(r-1) \left( r\Delta - \sum_{i=1}^l r_i \Delta_i \right) + \left( \sum_{i<j} r_i r_j (\mu_j - \mu_i) \cdot K_{\F} \right) \]
is bounded below by a constant $\kappa$ which depends only on $(r,c)$ and $r_i$ for all $ 1 \leq i \leq l$, except when $l=2$ and $\mu_F(\gamma_2) - \mu_F(\gamma_1) = -1$. Thus, we have 
\[ \Delta' \leq N' + 1 +  \frac{1}{2}\left(r^2 - \sum_{i=1}^l r_i^2 \right) - \frac{1}{2}\kappa  \]

We would like to scrutinize the special case when $l=2$ and $\mu_F (\gamma_2) - \mu_F(\gamma_1) = -1$. Note that it follows from Mozgovoy's theorem (Theorem \ref{hrc}) that $H_{r,c}$ only depends on whether or not $r \mid c \cdot F$. Let $r = r_1 + r_2$, $c = aE + bF$, $c_1 = r_1a_1 E + b_1 F$ and $c_2 = r_2 a_2 E + b_2 F$. We will denote $H_{r_i,c_i}$ by $H_{r_i}$ for $i=1,2$ because we are assuming that $r_i \mid c_i \cdot F$ for $i=1,2$. It follows from equation \ref{eqnSmu} that for $S^\mu(\gamma_1,\gamma_2;F,E+F)$ to be  nonzero, we must have $\mu_{E+F}(\gamma_1) \leq \mu_{E+F}(\gamma_2)$, or equivalently, we have $b_2 \geq \frac{br_2}{r}$. Furthermore, we see that 
\[ -\frac{1}{2} r_1 r_2(\mu_2 - \mu_1)\cdot K_{\F} = r_1 r_2 (a_2 - a_1) + rb_2 - r_2 b \] 
and 
\[ r\Delta - r_1 \Delta_1 - r_2 \Delta_2 = \frac{r_1 r_2 }{2r}(a_2 - a_1)^2 - (a_2 - a_1)b_2 + b\frac{r_2(a_2 - a_1)}{r}\]
Using these equations together with the fact that $a_2 - a_1 = -1$, we see that equation \ref{rhs} transforms to 
\[ (1-q)\L^{\frac{1}{2}(r^2 - r_1^2 - r_2^2)}\L^{-r_1 r_2}q^{\frac{r_1 r_2}{2r} - \frac{b r_2}{r}}q^{b_2} \prod_{i=1}^2 \L^{r_i^2}H_{r_i}(\L^{-(r+r_i)}q) 
\]
whenever $b_2 \geq \frac{br_2}{r}$ and is zero otherwise. Adding all these terms for $b_2 \geq \frac{br_2}{r}$  yields 
\begin{equation}\label{spcase}
 \L^{\frac{1}{2}(r^2 - r_1^2 - r_1^2) - r_1 r_2} q^{\frac{r_1 r_2}{2r} - \frac{b r_2}{r}} q^{\left\lceil \frac{b r_2}{r}\right\rceil} \prod_{i=1}^2 \L^{r_i^2}H_{r_i}(\L^{-(r+r_i)}q)
\end{equation}
Each nonzero term appearing in the coefficient of $\L^{-N'}q^{\Delta'}$ in equation \ref{spcase} arises from a pair of equations 
\begin{align*}
\Delta' &= \frac{r_1 r_2}{2r} - \frac{br_2}{r} + \left\lceil \frac{br_2}{r} \right\rceil + \Delta'_1 + \Delta'_2 \\
-N' &= \frac{1}{2}(r^2 - r_1^2 - r_2^2) - r_1 r_2 - N'_1 - N'_2 = -N'_1 - N'_2
\end{align*}
where $(\Delta'_i,N'_i) $ accounts for contribution coming from terms of coefficient of $\L^{-N'_i}q^{\Delta'_i}$ in $\L^{r_i^2}$ $ H_{r_i}(\L^{-(r+r_i)}q)$. We have shown before that we must have $\Delta'_i \leq N'_i$ for $i = 1,2$. Hence, we must have 
\[ \Delta' \leq N' + \frac{r_1 r_2}{2r} + \left( \left\lceil \frac{br_2}{r}\right\rceil - \frac{br_2}{r}\right) \]

In conclusion, we have 
\[ \Delta' \leq N' +  C_0 \]
where $C_0$ is the supremum of $0$, the terms $1 + \frac{1}{2}\left( r^2 - \sum_{i=1}^l r_i^2 \right) - \frac{1}{2}\kappa  $ corresponding to $l \geq 2$ and $r_1 + \cdots + r_l = r$, and the terms $\frac{r_1 r_2}{2r} + \left( \left\lceil \frac{br_2}{r}\right\rceil - \frac{br_2}{r}\right)$ corresponding to $l=2$, $r_1 + r_2 = r$, and $\mu_F(\gamma_2) - \mu_F(\gamma_1) = -1$.  

It follows from equation \ref{kapval} that $\kappa$ is bounded below by $-(r-1)$. Clearly, $\left( r^2 - \sum_{i=1}^l r_i^2 \right)$ is bounded above by $r^2 - r$. Hence, we see that
\[ 1 + \frac{1}{2} \left( r^2 - \sum_{i=1}^l r_i^2 \right) - \frac{1}{2}\kappa \;\leq\; \frac{1}{2}\left( r^2 + 1 \right)\]

Clearly $\left( \left\lceil \frac{br_2}{r} \right\rceil - \frac{br_2}{r}\right) \leq 1$ and $\frac{r_1(r-r_1)}{2r}$ is bounded above by $\frac{r}{8}$, whence the terms corresponding to $r=r_1 + r_2$ and $\mu_F(\gamma_2 ) - \mu_F(\gamma_1) = -1$ are bounded above by $\frac{r}{8}+1$.

In summary, we can take $C_0$ to be $\frac{1}{2}(r^2 + 1)$. Hence, for $\Delta' > N'  + \frac{1}{2}(r^2 + 1) $, the coefficient of $\L^{-N'}q^{\Delta'}$ in $(1-q)\tilde{G}_{r,c}(q)$ is zero.
\end{proof}
 
\begin{lem}\label{lowerbound}
The following expression 
\begin{equation}\label{lbterm}
 2(r-1) \left( r\Delta - \sum_{i=1}^l r_i \Delta_i \right) + \left( \sum_{i<j} r_i r_j (\mu_j - \mu_i) \cdot K_{\F} \right) 
\end{equation}
is bounded below by some constant $\kappa$ which depends only on $(r,c)$ and $r_i$ for $1 \leq i \leq l$, except when $l=2$ and $\mu_F(\gamma_2 ) - \mu_F(\gamma_1) = -1$.
\end{lem}
\begin{proof}
We can assume that $r_i \mid c_i \cdot F$ for each $1 \leq i \leq l$, otherwise the entire summand (equation \ref{rhs}) vanishes due to Mozgovoy's theorem (Theorem \ref{hrc}). Let $c = aE + bF$ and for each $1 \leq i \leq l$, let $c_i = r_i a_i E +  b_i F$. Note that every term in the generating function $\tilde{G}_{r,c}(q)$ is invariant under the action of tensoring by line bundles, whence, we can assume that $0 \leq a,b \leq (r-1)$. Furthermore, we define $s_i = \sum_{j=i}^l b_j$ for all $1 \leq i \leq l$.

Following Manschot \cite{man}[Proof of Proposition 4.1] we see that 
\begin{equation*}
\begin{aligned}
r\Delta - \sum_{i=1}^l r_i \Delta_i &= \sum_{i=2}^l \frac{r_i}{2 \left( \sum_{j=1}^i r_j \right) \left( \sum_{j=1}^{i-1} r_j \right)}\left( \sum_{j=1}^{i-1} r_j(a_i - a_j) \right)^2 \; - \sum_{i=2}^l (a_i - a_{i-1})s_i \;  \\
& \qquad \qquad \qquad \qquad + \, b \sum_{i=2}^l \frac{\sum_{j=1}^{i-1} r_i r_j(a_i - a_j)}{\left( \sum_{j=1}^i r_j \right) \left( \sum_{j=1}^{i-1} r_j \right)}
\end{aligned}
\end{equation*}
Similarly, following Manschot \cite{man}[Proof of Proposition 4.1] we see that
\begin{align*}
\sum_{i<j} r_i r_j (\mu_j - \mu_i) \cdot K_{\F} = \sum_{i<j} r_i r_j (a_i - a_j) -2 \sum_{i=2}^l (r_i + r_{i-1}) s_i + 2(r - r_1)b
\end{align*}
Using these two equations we get 
\begin{equation}\label{eqn16}
\begin{aligned}
& 2(r-1) \left( r\Delta - \sum_{i=1}^l r_i \Delta_i \right) + \left( \sum_{i<j} r_i r_j (\mu_j - \mu_i) \cdot K_{\F} \right) = \\
& \left\lbrace 2(r-1) \sum_{i=2}^l \frac{r_i}{2 \left( \sum_{j=1}^i r_j \right) \left( \sum_{j=1}^{i-1} r_j \right)}\left( \sum_{j=1}^{i-1} r_j(a_i - a_j) \right)^2 +\sum_{i<j} r_i r_j (a_i - a_j)   \right\rbrace \\
& \qquad \qquad \qquad + \left\lbrace 2(r-1)b \sum_{i=2}^l \frac{\sum_{j=1}^{i-1} r_i r_j(a_i - a_j)}{\left( \sum_{j=1}^i r_j \right) \left( \sum_{j=1}^{i-1} r_j \right)}  \right. \; + \\
& \left. -2(r-1) \sum_{i=2}^l (a_i - a_{i-1}) s_i - 2 \sum_{i=2}^l (r_i + r_{i-1}) s_i  + 2 (r-r_1) b \right\rbrace
\end{aligned}
\end{equation}

We would like to show that both the first and second summand of right hand side of equation \ref{eqn16} are bounded below. Let us call the first summand $S_1$ and the second summand $S_2$.

We now proceed to scrutinize $S_1$ to determine its lower bound. We are going to use the following identity of Manschot \cite{man}[Proof of Proposition 4.1] 
\begin{equation}\label{iden1}
\sum_{i=2}^l \frac{r_i}{2 \left( \sum_{j=1}^i r_j \right) \left( \sum_{j=1}^{i-1} r_j \right)}\left( \sum_{j=1}^{i-1} r_j(a_i - a_j) \right)^2 = \frac{1}{2r}\left( \sum_{i=1}^l r_i (r-r_i)a_i^2 - 2 \sum_{1 \leq i<j \leq l} r_i r_j a_i a_j \right)
\end{equation}
Since $a = \sum_{i=1}^l r_i a_i$, it follows from equation \ref{iden1} that 
\[ S_1 = (r-1)\sum_{i=1}^l r_i a_i^2 - \frac{r-1}{r} a^2 + \sum_{i=1}^l a_i r_i \left(\sum_{j=i+1}^l r_j - \sum_{j=1}^{i-1} r_j \right) \]
Consider the smooth polynomial function 
\[ f(x_1, \cdots, x_l) = \sum_{i=1}^l r_i x_i^2 - \frac{1}{r}a^2 + \sum_{i=1}^l x_i \frac{r_i}{r-1}\left( \sum_{j=i+1}^l r_j - \sum_{j=1}^{i-1}r_j\right) \]
Clearly, the Hessian of $f$, given by $\left( \frac{\partial^2 f}{\partial x_j \partial x_i} \right)$ is positive definite. We define 
\[ g(x_1, \cdots, x_l) = \sum_{i=1}^l r_i x_i - a \]
Our goal is to minimize $f$ along  the locus of $g=0$ for integer values of the $x_i$'s. Using the Lagrange's multiplier method, we see that $f$ assumes minima at 
\[ a_i = \frac{a}{r} - \frac{1}{2(r-1)}\left(\sum_{j=i+1}^l r_j - \sum_{j=1}^{i-1} r_j\right) , \qquad \text{ for } i=1, \cdots, l \]
Clearly $\left\vert \sum_{j=i+1}^l r_j - \sum_{j=1}^{i-1}r_j \right\vert \leq (r-1)$, and hence we get $\frac{a}{r}-\frac{1}{2} \leq a_i \leq \frac{a}{r} + \frac{1}{2}$ for all $1 \leq i \leq l$. Thus, to find a lower bound for $S_1$ we need to find the minimum value of $f$ when $x_i \in \left\lbrace -1,0,1,2 \right\rbrace$ for all $1 \leq i \leq l$. We have the following partition 
\[ \left\lbrace 1, \cdots, l \right\rbrace = \left\lbrace i_\alpha \right\rbrace_{1\leq \alpha \leq p} \cup \left\lbrace j_\beta \right\rbrace_{1\leq \beta \leq q} \cup \left\lbrace k_\gamma \right\rbrace_{1\leq \gamma \leq s} \cup \left\lbrace m_\delta \right\rbrace_{1 \leq \delta \leq t} \]
where $x_{i_\alpha} = -1$, $x_{j_\beta} = 1$, $x_{k_\gamma} = 2$, and $x_{m_\delta} = 0$. We see that 
\begin{align}\label{s1}
\begin{aligned}
r(r-1)f &= (12 r - 9) \left( \sum_{i_\alpha > k_\gamma} r_{i_\alpha} r_{k_\gamma} \right) + (6r-4)\left(\sum_{i_\alpha > j_\beta} r_{i_\alpha} r_{j_\beta} + \sum_{k_\gamma < m_\delta} r_{k_\gamma} r_{m_\delta}\right) \\
&  + (2r-1) \left( \sum_{i_\alpha >m_\delta} r_{i_\alpha}r_{m_\delta} + \sum_{j_\beta > k_\gamma} r_{j_\beta} r_{k_\gamma} + \sum_{j_\beta < m_\delta} r_{j_\beta}r_{m_\delta} \right)\\
&+ (6r-9) \left( \sum_{i_\alpha < k_\gamma} r_{i_\alpha}r_{k_\gamma} \right) + (2r-4)\left(\sum_{i_\alpha < j_\beta} r_{i_\alpha} r_{j_\beta} + \sum_{k_\gamma >m_\delta} r_{k_\gamma} r_{m_\delta} \right) \\
&+ (-1)\left( \sum_{i_\alpha < m_\delta}r_{i_\alpha}r_{m_\delta} + \sum_{j_\beta < k_\gamma}r_{j_\beta}r_{k_\gamma} + \sum_{j_\beta > m_\delta}r_{j_\beta}r_{m_\delta}\right)
\end{aligned}
\end{align}
Note that since $r\geq 2$ all the summands in equation \ref{s1} except the last one have non-negative coefficient. By further examining the summands with non-negative coefficient, we see that together they must be bounded below by $(2r-4)$ because all the inequalities in the summations cannot be simultaneously compatible. Moreover, the negative summand is bounded below by $-(r^2 - r)$. Hence, $S_1$ is bounded below by $-r+3 -\frac{4}{r}$.

Our next goal is to determine a lower bound for $S_2$. We are going to use the following identities of Manschot \cite{man}[Proof of Proposition 4.1] 
\begin{equation}\label{id2}
\sum_{i=2}^l \frac{r_i}{\left( \sum_{j=1}^i r_j \right) \left( \sum_{j=1}^{i-1}r_j \right)} \left( \sum_{j=1}^{i-1} r_j(a_i - a_j) \right) = \frac{1}{r}\left( \sum_{i=2}^l (a_i - a_{i-1})\left(\sum_{j=i}^l r_j \right)\right)
\end{equation}
and 
\begin{equation}\label{id3}
\sum_{i=2}^l (r_i + r_{i-1})\left(\sum_{j=i}^l r_j\right) = (r-r_1)r
\end{equation}
The identities in equations \ref{id2} and \ref{id3} yields
\[ S_2 = 2\sum_{i=2}^l \left( (r-1)(a_i - a_{i-1}) + (r_i + r_{i-1})\right) \left( \frac{b}{r}\left( \sum_{j=i}^l r_j \right) - s_i \right) \]
Following Coskun and Woolf \cite{cos}[Proof of Theorem 5.4], we interpret the definition of $S(\{\gamma_\bullet \};F,$ $ E+F)$  (equation \ref{defnSmu}) in our current situation, we obtain for all $2 \leq i \leq l$
\begin{align}\label{defnSmuIn}
\begin{aligned}
& \qquad \text{A) } (a_i - a_{i-1}) <0 \text{ and } s_i \geq \frac{b}{r}\left( \sum_{j=i}^l r_j  \right) \\
& \qquad \text{B) } (a_i - a_{i-1}) \geq 0 \text{ and } s_i < \frac{b}{r} \left( \sum_{j =i}^l r_j  \right) 
\end{aligned}
\end{align}
In Case A, we see that $(r-1)(a_i - a_{i-1}) + r_i + r_{i-1} \leq 0$ except when $l=2$ and $a_2 - a_1 = -1$, which is not possible by our assumption. Hence, the term 
\begin{equation}\label{s2term}
\left( (r-1)(a_i - a_{i-1}) + (r_i + r_{i-1})\right) \left( \frac{b}{r}\left( \sum_{j=i}^l r_j \right) - s_i \right)
\end{equation}
is non-negative. 

Similarly, in Case B, we see that $(r-1)(a_i - a_{i-1}) + r_i + r_{i-1} \geq (r_i + r_{i-1})$, hence the term in equation \ref{s2term} is non-negative. Additionally, by using the fact that $s_i$ are integers, it follows from equation \ref{defnSmuIn} that we have a slightly better bound of equation \ref{s2term} 
\[ \left\vert (r-1) (a_i - a_{i-1}) + r_i + r_{i-1} \right\vert \left( 1 - sgn\left(a_i - a_{i-1} + \frac{1}{2}\right) \left( 1 - 2 \left\lbrace - \frac{b}{r}\sum_{j=i}^l r_j \right\rbrace\right)\right) \]
where $sgn$ is the sign function and $\lbrace \bullet \rbrace$ is the fractional part of any real number.

In conclusion, we can take $\kappa$ to be 
\begin{align}\label{kapval}
\begin{aligned}
& -r + 3 - \frac{4}{r}\\
&+ \sum_{i=2}^l \left\vert (r-1) (a_i - a_{i-1}) + r_i + r_{i-1} \right\vert \left( 1 - sgn\left(a_i - a_{i-1} + \frac{1}{2}\right) \left( 1 - 2 \left\lbrace - \frac{b}{r}\sum_{j=i}^l r_j \right\rbrace\right)\right)
\end{aligned}
\end{align}
which is our lower bound for equation \ref{lbterm}.
\end{proof}

Now that we have shown that for $\tilde{\Delta} \gg \tilde{N}$, the coefficient of $\L^{-\tilde{N}}q^{\tilde{\Delta}}$ in $(1-q)\tilde{G}_{r,\tilde{c}}(q)$ vanishes (see Proposition \ref{boundMF}), our goal is to relate $G_{r,c}(q)$ with $\tilde{G}_{r,\tilde{c}}$ using the blow-up formula, and conclude a similar result for $G_{r,c}(q)$.
\section{\textsc{Estimating the generating function $G_{r,c}(q)$ when rank is at least two}}\label{section5}
In this section, our goal is to show that there is a constant $C$ depending only on $r$ and $c$ such that when $\Delta > N + C$, the coefficient of $\L^{-N}q^\Delta$ in $(1-q)G_{r,c}(q)$ is zero. To show this, we are going to look at the blow-up $\F \rarrow \P$ and use the blow-up formula due to Mozgovoy \cite{moz}[Proposition 7.3] to relate the generating functions $G_{r,c}(q)$ and $\tilde{G}_{r,\tilde{c}}(q)$ (see equation \ref{blowupeqn}) in $A^-$. We are going to scrutinize the terms appearing in this relation, and use Proposition \ref{boundMF} to derive our inequality (see Theorem \ref{clm14}).

Recall from section \ref{section2} that we have a blow-up $\F \rarrow \P$ at point $p \in \P$. Let $\gamma = (r,c,\Delta)$ be a Chern character on $\P$. Let $m$ be the multiplicity of $c$ at the point $p$. Let $\tilde{\gamma} = (r, c - mE, \tilde{\Delta})$ be a Chern character on $\F$. The blow-up formula due to Mozgovoy \cite{moz}[Proposition 7.3] is the following equation 
\begin{equation}
\label{blowup}
\sum_{ch_2} [\MF (r,c-mE,ch_2)] q^{-ch_2} = F_m(q) \sum_{ch_2} [\MP (r, c , ch_2)] q^{- ch_2}
\end{equation}
where 
\begin{equation}
\label{Fm}
F_m(q) = \left( \prod_{k=1}^\infty \frac{1}{(1 - \L^{rk}q^k)^r} \right) \left( \sum_{\substack{ \sum_{i=1}^r a_i = 0, \\ a_i \in \mathbb{Z} + \frac{m}{r}}} \L^{\sum_{i<j} \binom{a_j - a_i}{2}} q^{- \sum_{i<j} a_i a_j} \right)
\end{equation}
Note that on $\P$, we have $-ch_2(\gamma) = r \Delta - \frac{c^2}{2r}$, while on $\F$, we have $-ch_2(\tilde{\gamma}) = r \tilde{\Delta} - \frac{c^2}{2r} + \frac{m^2}{2r}$. Hence, we can rewrite the blow-up equation (equation \ref{blowup}) 
\begin{equation}\label{eqn19}
\sum_{\Delta \geq 0} [\MP (r, c, \Delta)] q^{r \Delta} = \frac{q^{\frac{m^2}{2r}}}{F_m(q)} \sum_{\tilde{\Delta} \geq 0} [\MF (r,c-mE,\tilde{\Delta})] q^{r\tilde{\Delta}}
\end{equation}
Replacing $q$ by $\L^{-2r}q$ and multiplying both sides by $\L^{r^2}$ in equation \ref{eqn19} yields
\begin{equation}\label{blowupeqn}
G_{r,c}(q) = \frac{(\L^{-2r}q)^{\frac{m^2}{2r}}}{F_m(\L^{-2r}q)} \tilde{G}_{r,c-mE}(q)
\end{equation}
It follows from equation \ref{blowupeqn} that in order to achieve our goal, we need to analyze $F_m(\L^{-2r}q)$ and find an estimate for $\Delta$ in this expression.

By examining the definition of $F_m$ in equation \ref{Fm}, we conclude that it depends only on the remainder of $m$ modulo $r$, which we shall denote by $\bar{m}$, which we will think of as an integer between $0$ and $r-1$. 

We see that 
\begin{equation}
\label{FmL}
F_{\bar{m}} (\L^{-2r}q) = \prod_{k=1}^\infty \frac{1}{(1 - \L^{-rk}q^k)^r} \,\sum_{\substack{ \sum_{i=1}^r a_i = 0, \\ a_i \in \mathbb{Z} + \frac{\bar{m}}{r}}} 
\L^{\sum_{i<j} \binom{a_j - a_i}{2} + 2r \sum_{i<j}a_i a_j } q^{- \sum_{i<j}a_i a_j}
\end{equation}

Since $\sum_{i=1}^r a_i = 0$, we see that 
\[ - \sum_{1 \leq i<j \leq r} a_i a_j = \frac{1}{2} \sum_{i=1}^r a_i^2 \]
and 
\[ \sum_{1 \leq i<j \leq r} \binom{a_j - a_i}{2} + 2r \sum_{1 \leq i<j \leq r} a_i a_j = - \frac{r}{2} \left(\sum_{i=1}^r a_i^2 \right) - \left(\sum_{i=1}^r i a_i \right) \]

We now use the following substitutions 
\begin{align*}
\qquad\qquad a_i &= b_i + \frac{\bar{m}}{r} \, , \text{ where } b_i \in \mathbb{Z}, \, \text{ for } 1 \leq i \leq r-1,  \\
\qquad\qquad a_r &= - \sum_{i=1}^{r-1} \left( b_i + \frac{\mb}{r} \right)
\end{align*}
These substitutions yield the following equations 
\begin{equation}\label{eqn23}
\begin{aligned}
- \frac{r}{2} \left(\sum_{i=1}^r a_i^2 \right) - \left(\sum_{i=1}^r i a_i \right) &= -r \left( - \frac{\mb^2}{2r} + \frac{\mb^2}{2} + \sum_{i=1}^{r-1} b_i^2 + \mb \sum_{i=1}^{r-1} b_i + \sum_{1 \leq i<j \leq (r-1)} b_i b_j \right) \\
& \qquad + \left( \frac{ (r-1) \mb}{2}  + \sum_{i=1}^{r-1} (r-i)b_i \right) \\
\frac{1}{2} \sum_{i=1}^r a_i^2 &= \left( \frac{- \mb^2}{2r} + \frac{\mb^2}{2} + \sum_{i=1}^{r-1} b_i^2 + \mb \sum_{i=1}^{r-1} b_i + \sum_{1 \leq i<j \leq (r-1)} b_i b_j \right)
\end{aligned}
\end{equation}

Employing the above equations \ref{eqn23} leads to the following expression for $F_{\mb}(\L^{-2r}q)$
\begin{align*}
F_{\mb}(\L^{-2r}q) &= \left(\prod_{k=1}^\infty \frac{1}{(1 - \L^{-rk}q^k)^r} \right) \left( \L^{-r}q \right)^{- \frac{(r+1) \mb^2}{2r} } \L^{\frac{ (r-1) \mb}{2}} \; \times \\
& \sum_{b_1, \cdots, b_{r-1} \in \mathbb{Z}} \L^{\sum_{i=1}^{r-1} (r-i)b_i} \left( \L^{-r}q \right)^{ \mb^2 + \sum_{i=1}^{r-1}b_i^2 + \mb \sum_{i=1}^{r-1} b_i + \sum_{i<j} b_i b_j}
\end{align*}

For sake of convenience, we define 
\begin{align}\label{Lambda}
\Lambda_d^{(\mb)} = \sum_{\substack{ b_1, \cdots, b_{r-1} \in \mathbb{Z}, \\ \mb^2 + \sum_{i=1}^{r-1} b_i^2 + \mb \sum_{i=1}^{r-1} b_i + \sum_{i<j} b_i b_j = d}} \L^{\sum_{j=1}^{r-1}(r-j)b_j}
\end{align}
Thus, we can think of the last  summation term of $F_{\mb}(\L^{-2r}q)$ as a power series
\begin{align}\label{FmLL}
F_{\mb}(\L^{-2r}q) = \left( \prod_{k=1}^\infty \frac{1}{(1 - \L^{-rk}q^k)^r} \right) \left( \L^{-r}q \right)^{ - \frac{(r+1)\mb^2}{2r}} \L^{\frac{(r-1)\mb}{2}} \left( \sum_{d=0}^\infty \Lambda_d^{(\mb)} (\L) \left( \L^{-r}q \right)^d \right)
\end{align}
\begin{rem}\label{rmk8}
Recall that any power series of the form $f(x) = 1 + a_1 x + a_2 x^2 + \cdots$ is invertible, and its inverse is given by $1 + b_1 x + b_2 x^2 + \cdots$, where for any positive integer $n$, we have 
\[ b_n = \sum_{\substack{n_1 + \cdots + n_l = n \\ n_i \in \mathbb{Z}_{>0}}} (-1)^l a_{n_1} \cdots a_{n_l} \]
\end{rem}
To analyze $G_{r,c}(q)$, we need to invert $F_{\mb}(\L^{-2r}q)$ (equation \ref{blowupeqn}), and a posteriori, we need to invert the power series $\sum_{d =0}^\infty \Lambda_d^{(\mb)}(\L) (\L^{-r}q)^d$. To do this, we need to figure out the least non-negative integer $d$ such that $\Lambda_d^{(\mb)}(\L)$  is nonzero. 

\begin{lem}\label{clm9}
The smallest non-negative integer $d$ for which $\Lambda_d^{(\mb)}$ is nonzero, is $\frac{\mb^2 + \mb}{2}$. Additionally, 
\[ \Lambda_{\frac{\mb^2 + \mb}{2}}^{(\mb)} (\L) = \L^{-r \mb} \; \sum_{\nu = \frac{\mb^2 + \mb}{2}}^{r \mb - \frac{\mb^2 - \mb}{2}} \rho_\nu \L^{\nu} \]
where $\rho_\nu$ is the cardinality of the set $\left\lbrace \left(j_1, \cdots, j_{\mb} \right) \,\vert\, 1 \leq j_1 < \cdots < j_{\mb} \leq r, \, j_1 + \cdots + j_{\mb} = \nu \right\rbrace$, when $\nu$ is a positive integer, and $\rho_0 = 1$.
\end{lem}
\begin{proof}
Note that 
\[ \mb^2 + \sum_{i=1}^{r-1} b_i^2 + \mb \sum_{i=1}^{r-1} b_i + \sum_{i<j} b_i b_j = \frac{1}{2} \left( \mb^2 + \sum_{i=1}^{r-1}b_i^2 + \left( \mb + \sum_{i=1}^{r-1} b_i \right)^2 \right) \]
Consequently, we need to figure out the smallest value of $\mb^2 + \sum_{i=1}^{r-1} b_i^2 + \left( \mb + \sum_{i=1}^{r-1} b_i \right)^2$, where $b_i \in \mathbb{Z}$ for all $1 \leq i \leq r-1$. 

If $\mb = 0$, we see that the equation $\sum_{i=1}^{r-1} b_i^2 + \left( \sum_{i=1}^{r-1} b_i \right)^2 = 0$ has only one solution, the trivial one. Thus, $\Lambda_0^{(0)}(\L) = 1$.

Assume $1 \leq \mb \leq r-1$. It follows from Lemma \ref{lowerbound2} (below), that the smallest value assumed by the expression $\sum_{i=1}^{r-1} b_i^2 + \left( \mb + \sum_{i=1}^{r-1} b_i \right)^2$ occurs at $b_1 = \cdots = b_{r-1} = - \frac{\mb}{r}$. As a result, we need to evaluate the expression when $b_i \in \left\lbrace -1,0 \right\rbrace$ for all $1 \leq i \leq r-1$, to figure out the minimum value of the expression for integer values. Suppose $k$ of the $b_i$'s are $(-1)$ and the remaining are zero, the expression becomes $k + \left( \mb -k \right)^2$. Clearly, the minimum value of $k + \left( \mb - k \right)^2$ for integer values of $k$ is $\mb$, which occurs when $k = \mb-1, \mb$.

In summary, when $1 \leq \mb \leq r-1$, the smallest value of the expression \[\frac{1}{2}\left( \mb^2 + \sum_{i=1}^{r-1} b_i^2 + \left( \mb + \sum_{i=1}^{r-1} b_i \right)^2 \right)\] for integer values of $b_i$ is $\frac{\mb^2 + \mb}{2}$, which occurs when $\mb-1$ or $\mb$ of the $b_i$'s are $(-1)$ and the remaining are zero. Hence, we have 
\[ \Lambda_d^{(\mb)} (\L) = \sum_{1 \leq j_1 < \cdots < j_{\mb -1} \leq r-1} \L^{j_1 + \cdots + j_{\mb-1} - (\mb-1)r} + \sum_{1 \leq j_1 < \cdots < j_{\mb} \leq r-1} \L^{j_1 + \cdots + j_{\mb} - r \mb} \]
Factoring out $\L^{-r \mb}$ leads to 
\[ \Lambda_d^{(\mb)} (\L) = \L^{-r\mb} \sum_{1 \leq j_1 < \cdots < j_{\mb} \leq r} \L^{j_1 + \cdots + j_{\mb}} \]
\end{proof}

Before proceeding further, we need to tie the loose ends of Lemma \ref{clm9} by analyzing the real valued polynomial function $y_1^2 + \cdots + y_n^2 + \left( A + y_1 + \cdots + y_n \right)^2$.
\begin{lem}\label{lowerbound2}
Consider the smooth real valued function 
\[ f(y_1, \cdots, y_n) = y_1^2 + \cdots + y_n^2 + \left( A + y_1 + \cdots + y_n \right)^2 \]
where $A$ is any real number. The Hessian of $f$ is positive definite. Furthermore, the function $f$ has a global minima at $y_1 = \cdots = y_n = - \frac{A}{n+1}$, and the minimum value for $f$ is $ \frac{A^2}{n+1}$.
\end{lem}
\begin{proof}
Clearly, we see that for $1 \leq k \leq n$
\[  \frac{\partial f}{\partial y_k} = 2y_k + 2\left( A + y_1 + \cdots + y_n \right) \]
Subsequently, we see that for $1 \leq l \leq n$
\[ \frac{\partial^2 f}{\partial y_l \partial y_k} = \begin{cases} 2, \text{ if } k \neq l \\ 
4, \text{ if } k = l \end{cases} \]
Let $H$ be the $n \times n$ matrix with $H_{l,k} = \frac{\partial^2 f}{\partial y_l \partial y_k}$, then we see that 
\[ \left( y_1 \; \cdots \; y_n \right) \cdot H \cdot \left( y_1 \; \cdots \; y_n \right)^{T} = 2 \left(\sum_{i=1}^n y_i^2 \right) + 2 \left( \sum_{i=1}^n y_i \right)^2 \]
Thus, $H$ is positive definite. As a consequence, $f$ has a global minimum when $\frac{\partial f}{\partial y_k} = 0$ for all $1 \leq k \leq n$. This system of linear equations has a unique solution $y_1 = \cdots = y_n = -\frac{A}{n+1} $. It follows that the minimum value for $f$ is $\frac{A^2}{n+1}$.
\end{proof}

Returning back to our track, we still need to analyze $F_{\mb}(\L^{-2r}q)$. Using equation \ref{FmLL} and Lemma \ref{clm9}, we see that 
\begin{align*}
F_{\mb}(\L^{-2r}q) &= \left( \prod_{k=1}^\infty \frac{1}{(1 - \L^{-rk}q^k)^r} \right) \left(\L^{-r}q \right)^{- \frac{(r+1) \mb^2}{2r}} \L^{ \frac{(r-1) \mb}{2}} \\
& \qquad \qquad\Lambda_{\frac{\mb^2 + \mb}{2}}^{(\mb)} (\L) \left( \L^{-r}q \right)^{\frac{\mb^2 + \mb}{2}} 
\sum_{d=0}^\infty \tilde{\Lambda}_{d}^{(\mb)} (\L) \left( \L^{-r}q \right)^d
\end{align*}
where $\tilde{\Lambda}_d^{(\mb)}(\L) = \left( \Lambda_{\frac{\mb^2 + \mb}{2}}^{(\mb)}(\L)\right)^{-1} \cdot \Lambda_{d + \frac{\mb^2 + \mb}{2}}^{(\mb)}(\L)$.

Finally, using remark \ref{rmk8}, we can invert $F_{\mb}(\L^{-2r}q)$.
\begin{equation}
\label{FmLinv}
\begin{aligned}
\left(F_{\mb} (\L^{-2r}q)\right)^{-1} &= \left( \prod_{k=1}^\infty (1 - \L^{-rk}q^k)^r \right) \left( \L^{-r}q \right)^{-\frac{r \mb - \mb^2}{2r}} \L^{-\frac{(r-1)\mb}{2}} \left( \Lambda_{\frac{\mb^2 + \mb}{2}}^{(\mb)} (\L) \right)^{-1} \\
& \qquad \qquad
\left( 1 + \sum_{d =1}^\infty \left( \sum_{\substack{d_1, \cdots, d_l \in \mathbb{Z}_{>0} \\ d_1 + \cdots + d_l = d}}(-1)^{l} \prod_{i=1}^l \tilde{\Lambda}_{d_i}^{(\mb)} \right) \left( \L^{-r}q\right)^{d} \right)
\end{aligned}
\end{equation}

Before tackling $G_{r,c}(q)$, we would like to analyze $F_{\mb}(\L^{-2r}q)^{-1}$ and produce bounds for $\Delta$ such that the coefficient of $\L^{-N}q^\Delta$ vanishes.
\begin{lem}\label{clm11}
If $\Delta > N - \frac{(r-\mb)\mb}{2r}$, then the coefficient of $\L^{-N}q^\Delta$ in $F_{\mb}(\L^{-2r}q)^{-1}$ is zero.
\end{lem}
\begin{proof}
We are going to produce an expression for $\left(\Lambda_{\frac{\mb^2 + \mb}{2}}^{(\mb)}(\L)\right)^{-1}$, and use it alongwith the expression for $F_{\mb}(\L^{-2r}q)^{-1}$ (see equation \ref{FmLinv}) to determine the bound for $\Delta$.

Using Lemma \ref{clm9} and factoring $\L^{r \mb - \frac{\mb^2 - \mb}{2}}$, we get
\[ \Lambda_{\frac{\mb^2 + \mb}{2}}^{(\mb)}(\L) = \L^{-\frac{\mb^2 - \mb}{2}} \,\sum_{\nu = 0}^{-(r\mb - \mb^2)} \rho_{\nu + r\mb - \frac{\mb^2 - \mb}{2}} \,\L^{\nu} \]
In a similar fashion as in remark \ref{rmk8}, it follows that 
\[ \left(\Lambda_{\frac{\mb^2 + \mb}{2}}^{(\mb)}(\L)\right)^{-1} = \L^{\frac{\mb^2 - \mb}{2}}\left( 1+ \sum_{\nu = -1}^{-\infty} \left( \sum_{\substack{\nu_1 , \cdots, \nu_l \in \mathbb{Z}_{<0} \\ \nu_1 + \cdots + \nu_l = \nu}} (-1)^l \prod_{i=1}^l \rho_{\nu_i + r\mb - \frac{\mb^2 - \mb}{2}} \right) \L^{\nu} \right)\]

It follows from equation \ref{FmLinv} that 
{\small 
\begin{align*}
\left(F_{\mb}(\L^{-2r}q)\right)^{-1} &=  \prod_{k=1}^\infty \left(\sum_{\alpha=0}^\infty (-1)^\alpha \binom{r}{\alpha} \L^{-rk\alpha}q^{k \alpha} \right) \times \left( \L^{-r}q \right)^{-\frac{r \mb - \mb^2}{2r}} \L^{-\frac{(r-1)\mb}{2}} \times 
\left( \Big(\Lambda_{\frac{\mb^2 + \mb}{2}}^{(\mb)}(\L)\Big)^{-1} \right.  \\
& \left. 
+ \sum_{d = -1}^{-\infty} \left( \sum_{\substack{d_1, \cdots, d_l \in \mathbb{Z}_{<0} \\ d_1 + \cdots + d_l = d}} (-1)^l \left(\Lambda_{\frac{\mb^2 + \mb}{2}}^{(\mb)}(\L)\right)^{-(l+1)} \prod_{i=1}^l \Lambda_{d_i + \frac{\mb^2 + \mb}{2}}^{(\mb)}(\L)
\right) \left(\L^{-r}q\right)^d
\right)
\end{align*}}

Each nonzero term appearing in the co-efficient of $\L^{-N}q^{\Delta}$ in $F_{\mb}(\L^{-2r}q)^{-1}$ arises from a pair of equations 
\begin{align*}
\Delta &= \left(\sum_{j=1}^\delta k^{(j)} \alpha^{(j)}\right) - \left(\frac{r\mb - \mb^2}{2r}\right) + d \\
-N &= \left(\sum_{j=1}^\delta -r k^{(j)} \alpha^{(j)}\right) + r \left( \frac{r\mb - \mb^2}{2r}\right) - \frac{(r-1)\mb}{2} + \\
& \;\qquad \qquad \left( \left( \frac{\mb^2 - \mb}{2} (l+1) + \sum_{i=1}^{l+1} \nu_i \right) + \sum_{i=1}^l \sum_{j=1}^{r-1} (r-j)b_j^{(i)} \right) - rd
\end{align*}
where the $\alpha$'s, the $\nu$'s, and $l$ are non-negative integers; the $k$'s are positive integers; and the $b_j^{(i)}$'s are integers satisfying 
\[ \mb^2 + \sum_{j=1}^{r-1} \left( b_j^{(i)} \right)^2 + \left( \mb + \sum_{j=1}^{r-1} b_j^{(i)} \right)^2 = 2d_i + \mb^2 + \mb , \qquad \text{ for } 1 \leq i \leq l \]
Subsequently, we will show (in Lemma \ref{clm12}) that $\left(\sum_{j-1}^{r-1} (r-j) b_j^{(i)}\right) + \frac{\mb^2 - \mb}{2} \leq (r-1)d_i$, for all $1 \leq i \leq l$. Consequently, we have 
\[ \left(\sum_{i=1}^l \sum_{j=1}^{r-1} (r-j) b_j^{(i)}\right) + \frac{\mb^2 -\mb}{2}l \leq (r-1) d \]
Therefore, we see that 
\[ N + (r-1)\frac{r\mb - \mb^2}{2r} - \frac{(r-1)\mb}{2} + \frac{\mb^2-\mb}{2} \geq \left(\sum_{j-1}^\delta k^{(j)} \alpha^{(j)}\right) - \frac{r \mb - \mb^2}{2r} + d  = \Delta \]
and hence, 
\[ N - \frac{(r-\mb)\mb}{2r} \geq \Delta \]
\end{proof}
Before we continue, we need to wrap up the proof of Lemma \ref{clm11} by proving the following:
\begin{lem}\label{clm12}
Let $d$ be a non-negative integer, and $\mb $ be a non-negative integer less than $r$. Suppose  $b_1, \cdots, b_{r-1}$ are integers satisfying 
\begin{equation}\label{eqn27}
 \mb^2 + \sum_{j=1}^{r-1} b_j^2 + \left( \mb + \sum_{j=1}^{r-1} b_j \right)^2 = 2d + \mb^2 +\mb 
\end{equation}
Then, we have $\sum_{j-1}^{r-1} (r-j) b_j \leq (r-1)d$.

Furthermore, if $r \geq 3$ and $2 \leq \mb \leq (r-1)$, then we have 
\[ \left(\sum_{j=1}^{r-1} (r-j)b_j\right) + \frac{\mb^2 - \mb}{2} \leq (r-1)d \]
\end{lem}
\begin{proof}
Before we begin the proof of Lemma \ref{clm12}, note that 
\begin{rem}\label{ordering}
Let $r_1, \cdots, r_n$ be positive integers satisfying $r_1 > \cdots > r_n$, and let $b_1 , \cdots , b_n$ be integers satisfying $b_1 \geq \cdots \geq b_n$. Let $\sigma $ be any permutation of $\left\lbrace 1, \cdots, n \right\rbrace$. Then, we have 
\[ r_1 b_{\sigma(1)} + \cdots + r_n b_{\sigma(n)} \leq r_1 b_1 + \cdots + r_n b_n \]
\end{rem}
Thus, if $b'_1, \cdots , b'_{r-1}$ be a rearrangement of $b_1, \cdots, b_{r-1}$ satisfying $b'_1 \geq \cdots \geq b'_{r-1}$, then we see that
\[ \sum_{j=1}^{r-1} (r-j) b_j \,\leq\, \sum_{j=1}^{r-1} (r-j) b'_j \]
Moreover, let $n_1, \, n_2,\, n_3$ be non-negative integers such that
\begin{enumerate}[$\qquad\qquad \bullet$]
\item $b'_{j_1} \geq \cdots \geq b'_{j_{n_1}} \geq 2$,
\item $b'_{j_{n_1 + 1}} = \cdots = b'_{j_{n_1 + n_2}} = 1 $, 
\item $-1 \geq b'_{j_{n_1 + n_2 + 1}} \geq \cdots \geq b'_{j_{n_1 + n_2 +n_3}}$, and 
\item $b'_j = 0$ for all $j \neq j_l, \, 1 \leq l \leq n_1 + n_2 + n_3$.
\end{enumerate}
Therefore, we have 
\[ \sum_{j=1}^{r-1}(r-j) b'_j \,\leq\, \sum_{l=1}^{n_1} (r-j_{l}) b'_{j_{l}} + \sum_{l=n_1 +1}^{n_1 + n_2} (r - j_{l}) \leq \frac{(r-1)}{2} \left( 2 \left( \sum_{l=1}^{n_1} b'_{j_l} \right) + 2n_2 \right) \] 
We observe that to complete our proof it is enough to show that 
\[ \left(\sum_{l=1}^{n_1} 2b'_{j_l}\right) + 2n_2 \leq 2d \]
Since $\left(b'_{j_l}\right)^2 \geq 2 b'_{j_l}$ for $1 \leq l \leq n_1$ and $\left(b'_{j_l}\right)^2 = 1$ for $n_1 + 1 \leq l \leq n_1 + n_2$, it follows from equation \ref{eqn27} that it is enough to show that 
\[ n_2 + \mb \leq \sum_{l=n_1 + n_2 +1}^{n_1 + n_2 + n_3} \left( b'_{j_l}\right)^2 + \left( \left( \mb + n_2 + \sum_{l=1}^{n_1} b'_{j_l}\right) + \sum_{l=n_1 + n_2 +1}^{n_1 + n_2 + n_3} b'_{j_l} \right)^2 \]
If $n_2 + \mb \leq n_3$, then we are done because $\left(b'_{j_l}\right)^2 \geq 1$ for all $n_1 + n_2 +1 \leq l \leq n_1 + n_2 + n_3$. Otherwise, it follows from Lemma \ref{lowerbound2} that 
\[ \sum_{l=n_1 + n_2 +1}^{n_1 + n_2 + n_3} \left( b'_{j_l}\right)^2 + \left( \left( \mb + n_2 + \sum_{l=1}^{n_1} b'_{j_l}\right) + \sum_{l=n_1 + n_2 +1}^{n_1 + n_2 + n_3} b'_{j_l} \right)^2 \geq \frac{1}{n_3 +1} \left( \mb + n_2 + \sum_{l=1}^{n_1} b'_{j_l} \right)^2 \]
Since $b'_{j_l} \geq 2$ for $1 \leq l \leq n_1$ and $n_2 + \mb \geq n_3 +1$, we have 
\[ \frac{1}{n_3 +1} \left( \mb + n_2 + \sum_{l=1}^{n_1} b'_{j_l} \right)^2 \geq n_2 + \mb \]

Now we are going to specialize to the case when $r \geq 3$ and $2 \leq \mb \leq r-1$. Clearly, since $\mb \geq 2$, we see that $\frac{\mb^2 - \mb}{2} = 1 + \cdots + (\mb-1)$. We define 
\begin{align*}
b'_j = \begin{cases} b_j, &\text{ if }\; 1 \leq j \leq (r - \mb)\\ b_j +1 , &\text{ if }(r - \mb +1 )\leq j \leq (r-1) \end{cases}
\end{align*}
As a consequence, we see that 
\[ \left(\sum_{j=1}^{r-1} (r-j) b_j\right) + \frac{\mb^2 - \mb}{2} = \sum_{j=1}^{r-1}(r-j)b'_j \] 
Additionally, we can rewrite equation \ref{eqn27} in terms of $b'_j$'s as follows 
\[ \sum_{j=1}^{r-1}\left(b'_j\right)^2 + \left( \sum_{j=1}^{r-1}b'_j \right)^2 + 2 \left( \sum_{j=1}^{r - \mb} b'_j \right) = 2d \]
As a result, to prove our claim, it is enough to show that 
\[ \frac{(r-1)}{2}\left\lbrace \sum_{j=1}^{r-1}\left(b'_j\right)^2 + \left( \sum_{j=1}^{r-1}b'_j \right)^2 + 2 \left( \sum_{j=1}^{r - \mb} b'_j \right) \right\rbrace - \left( \sum_{j=1}^{r-1}(r-j)b'_j \right) \geq 0 \]
for integer values of $b'_j$, for all $1 \leq j \leq r-1$. Consider the smooth polynomial function 
\[ f(x_1, \cdots, x_{r-1}) = \frac{(r-1)}{2}\left\lbrace \sum_{j=1}^{r-1}x_j^2 + \left( \sum_{j=1}^{r-1}x_j \right)^2 + 2\left(\sum_{j=1}^{r-\mb}x_j \right)\right\rbrace - \left( \sum_{j=1}^{r-1} (r-j)x_j \right) \]
We have 
\begin{align*}
\frac{\partial f}{\partial x_k} = \begin{cases} \frac{(r-1)}{2}\left\lbrace 2x_k + 2\left( \sum_{j=1}^{r-1} x_j\right) + 2 \right\rbrace - (r-k), &\text{ if }\; 1 \leq k \leq (r-m) \\
\frac{(r-1)}{2}\left\lbrace 2x_k + 2 \left(\sum_{j=1}^{r-1} x_j\right)\right\rbrace - (r-k), &\text{ if }\; (r- \mb +1) \leq k \leq (r-1) \end{cases}
\end{align*}
and, the second partial derivatives are 
\begin{align*}
\frac{\partial^2 f}{\partial x_l \partial x_k} = \begin{cases}  2\frac{(r-1)}{2}, &\text{ if }\; l \neq k\\
4\frac{(r-1)}{2}, &\text{ if }\; l=k \end{cases}
\end{align*}
Since $r \geq 3$ and the Hessian matrix for $f$ is $\frac{(r-1)}{2}$ times the Hessian matrix in Lemma \ref{lowerbound2}, we conclude that our Hessian matrix is positive definite. Thus, $f$ has a global minimum at the critical point 
\begin{align*}
x_k = \begin{cases} -\frac{\mb}{r} -\frac{1}{2} + \frac{(r-k)}{(r-1)}, &\text{ if }\; 1 \leq k\leq (r-\mb)\\
-\frac{\mb}{r} + \frac{1}{2} + \frac{(r-k)}{(r-1)}, &\text{ if }\; (r-\mb+1) \leq k \leq (r-1)\end{cases}
\end{align*}
It follows from the bounds on $k$ that in either case, we have $-\frac{1}{2} \leq x_k \leq \frac{1}{2}$. Hence, to show that $f$ is non-negative for all integer values of $x_j$, for all $1 \leq j \leq (r-1)$, it is enough to show that $f$ is non-negative for every element of the set $\left\lbrace -1,0,1 \right\rbrace^{r-1}$. Let $(x_1, \cdots, x_{r-1})$ be an element of the set $\left\lbrace -1,0,1 \right\rbrace^{r-1}$. Furthermore, assume that for $1 \leq j \leq (r-\mb)$, $x$ of the $x_j$'s are $(+1)$ and $y$ of the $x_j$'s are $(-1)$. On a similar note, assume that for $(r-\mb +1) \leq j\leq (r-1)$, $z$ of the $x_j$'s are $(+1)$ and $w$ of the $x_j$'s are $(-1)$. It follows from Remark \ref{ordering} that 
\begin{align*} 
\sum_{j=1}^{r-1}(r-j)x_j &\leq (r-1) + \cdots + (r - x) - \left\lbrace \mb + (\mb+1) + \cdots + (\mb+y-1)\right\rbrace \\
&\qquad \qquad + (\mb-1) + \cdots + (\mb - z) - \left\lbrace 1 + \cdots + w \right\rbrace \\
&= rx - \mb y + \mb z - \frac{x^2 +x}{2} - \frac{y^2 -y}{2} - \frac{z^2 +z}{2} - \frac{w^2 +w}{2}
\end{align*}
Therefore, we have 
\begin{align}\label{eqn29}
\begin{aligned}
f(x_1, \cdots, x_{r-1}) &\geq \frac{(r-1)}{2} \left\lbrace (x-y+z-w)^2 + 3x -y + z + w \right\rbrace\\
& - \left\lbrace rx - \mb y + \mb z - \frac{x^2 +x}{2} - \frac{y^2 -y}{2} - \frac{z^2 +z}{2} - \frac{w^2 +w}{2} \right\rbrace
\end{aligned}
\end{align}
For ease of notation, let's call the right hand side of inequality in equation \ref{eqn29} as $g(x,y,z,w)$. Upon further scrutinizing, we deduce that 
\[ 2g(x,y,z,w) = (r-1)(x-y+z-w)^2 + (x^2 + y^2 + z^2 + w^2) + (r-2)x + (2\mb -r)y + (r-2\mb)z + rw \]
If $r = 2\mb$, then $2g(x,y,z,w) \geq 0$ because $x$ and $w$ are non-negative integers. If $r>2\mb$, then we see that 
\[ 2g(x,y,z,w) \geq (r - 2\mb)\left\lbrace (x-y+z-w)^2 + (x-y+z-w) \right\rbrace \geq 0 \]
Similarly, if $r<2\mb$, then using the fact that $(r-1)> (2\mb -r)$, we get 
\[ 2g(x,y,z,w) \geq (2\mb -r) \left\lbrace (-x+y-z+w)^2 + (-x+y-z+w)\right\rbrace \geq 0 \]
In conclusion, the function $f$ is non-negative for all integer values of $x_j$, for all $1 \leq j \leq r-1$.
\end{proof} 

We are finally ready to analyze $(1-q)G_{r,c}(q)$.
\begin{thm}\label{clm14}
If $\Delta > N  + \frac{(2 - 2r)\mb^2 - r\mb}{2r} + C_0$, where $C_0$ is the same constant as in Proposition \ref{boundMF}, then the coefficient of $\L^{-N}q^{\Delta}$ in $(1-q)G_{r,c}(q)$ is zero.
\end{thm}
\begin{proof}
Recall that if follows from the blow-up equation (equation \ref{blowupeqn}) that 
\[ (1-q)G_{r,c}(q) = \left(\L^{-2r}q \right)^{\frac{m^2}{2r}} \times  \left(F_m(\L^{-2r}q)\right)^{-1} \times (1-q)\tilde{G}_{r,c-mE}(q) \]
Each nonzero term appearing in the co-efficient of $\L^{-N}q^{\Delta}$ arises from a pair of equations 
\begin{align*}
\Delta &= \frac{\mb^2}{2r} + \Delta_1 + \Delta_2 \\
-N &= -\mb^2 + \left(-N_1 \right) + \left( -N_2 \right)
\end{align*}
where $\left( \Delta_1, -N_1 \right)$ accounts for the contribution of terms from the co-efficient of $\L^{-N_1}q^{\Delta_1}$ in $\left(F_{\mb}(\L^{-2r}q)\right)^{-1}$, and $\left( \Delta_2, -N_2 \right)$ accounts for the contribution of terms from the co-efficient of $\L^{-N_2}q^{\Delta_2}$ in $(1-q)\tilde{G}_{r,c - m E}(q)$.

It follows from Lemma \ref{clm11} and Proposition \ref{boundMF} that
\begin{align*}
\Delta_1 \leq N_1 - \frac{(r-\mb)\mb}{2r}\,,\qquad \text{ and } \;\,
\Delta_2 \leq N_2 + C_0
\end{align*}
These inequalities yield
\[ \Delta \leq N +   \frac{(2 - 2r)\mb^2 - r\mb}{2r} + C_0 \]
In conclusion, for $\Delta > N + \frac{(2 - 2r)\mb^2 - r\mb}{2r} + C_0 $, the co-efficient of $\L^{-N}q^\Delta$ in $(1-q) G_{r,c}(q)$ is zero.
\end{proof}

\section{\textsc{Bounds for stabilization of Betti numbers}}\label{section6}
In this section, our goal is to determine lower bounds such that the Betti numbers of the moduli space stabilize. More precisely, we look at $\P$ equipped with the ample divisor $H = c_1( \mathcal{O}_{\mathbb{P}^2}(1))$. We assume that $r$ and $a$ are coprime and consider the moduli space $M_{\mathbb{P}^2,H}(r,aH,c_2)$. Since $r$ and $a$ are coprime, all $\mu_H$-semistable sheaves are $\mu_H$-stable. Using Proposition \ref{proposition10} in conjunction with Theorem \ref{clm14}, we derive the lower bounds such that the Betti numbers of $M_{\P,H}(r,aH,c_2)$ stabilize. Lastly, we investigate some examples and show that we can improve this bound further.
\begin{thm}\label{theorem26}
Let $r$ be at least two. Assume that $r$ and $a$ be coprime. There is a constant $C$ depending only on $r$ and $a$ such that if $c_2 \geq N + C$, the $2N$th Betti number of the moduli space $M_{\P,H}(r,aH,c_2)$ stabilize. Moreover, we can take $C = \left\lfloor \frac{r-1}{2r}a^2 + \frac{1}{2}(r^2 + 1) \right\rfloor$.
\end{thm}
\begin{proof}
Let $\gamma$ denote the Chern class $(r,aH, c_2)$. By our assumption, $r$ and $a$ are coprime, a posteriori, all $\mu_H$-semistable sheaves are $\mu_H$-stable. In this case, we know that $M_{\P,H}(\gamma)$ is a smooth projective variety of dimension $ext^1(\gamma,\gamma)$. We conclude using Remark \ref{remark8} that to show that the $2N$th Betti number stabilize for $c_2 \geq N + C$, it is enough to show that the coefficient of $\L^{-N}q^d$ in the generating function 
\[ (1-q) \sum_{c_2 \geq 0}[M_{\P,H}(\gamma)] \L^{-ext^1(\gamma,\gamma)}q^{c_2} \]
is zero for $d > N + C$. 

We note that $\chi(\gamma,\gamma) = 1 - ext^1(\gamma,\gamma)$ and $ c_2 = r\Delta + \frac{r-1}{2r} c_1^2$. Proposition \ref{proposition10} yields the following equality in $A$
\[ [M_{\P,H}(r,aH,c_2)] = (\L -1) [\MP (r,aH,c_2)] \]
Thus, we have the following equality of generating functions 
\[ (1-q) \sum_{c_2 \geq 0 } [M_{\P,H}(\gamma)]\L^{-ext^1(\gamma,\gamma)} q^{c_2} = q^{\frac{r-1}{2r}a^2} (1 - \L^{-1})(1-q)G_{r,aH}(q) \]
Each term contributing to the coefficient of $\L^{-N}q^d$ in $q^{\frac{r-1}{2r}a^2} (1 - \L^{-1})(1-q)G_{r,aH}(q) $ arises from a pair of equations 
\begin{align*}
d &= \frac{r-1}{2r}a^2 + \Delta' \\
-N &= \varepsilon - N'
\end{align*}
where $\varepsilon \in \{-1,0 \}$ accounts for the contribution to the coefficient of $\L^{-N}q^d$ coming from $(1-\L^{-1})$, and $(\Delta',N')$ accounts for the contribution coming from the coefficient of $\L^{-N'}q^{\Delta'}$ in $(1-q)G_{r,aH}(q)$. It follows from Theorem \ref{clm14} that for the coefficient of $\L^{-N'}q^{\Delta'}$ to be nonzero, we must have $\Delta' \leq N' + C_0$ (using $m = 0$). Moreover, it follows from Proposition \ref{boundMF} that we can take $C_0 = \frac{1}{2}(r^2 + 1)$. Consequently, for the coefficient of $\L^{-N}q^d$ in $q^{\frac{r-1}{2r}a^2} (1 - \L^{-1})(1-q)G_{r,aH}(q)$ to be nonzero, we must have 
\[ d \leq N +  \left\lfloor \frac{r-1}{2r}a^2 + C_0 \right\rfloor \]
\end{proof}

For the remainder of this section, we look at some examples. Yoshioka \cite{yos94}[Page 194] has computed the Betti numbers $b_{2N}(M_{\P,H} (2,-H,c_2))$, where $M_{\P,H}(2, -H, c_2)$ is the moduli space of $\mu_H$-stable sheaves with Chern classes $(2,-H, c_2)$, which we will denote by $\gamma$. We observe from the table in \cite{yos94}[Page 194] that the Betti numbers $b_{2N}(M_{\P,H} (\gamma))$ stabilize when $c_2 \geq N+1$. Since $r=2$ and $a = -1$, we get from Theorem \ref{theorem26} that the Betti numbers stabilize when $c_2 \geq N + 2$. Therefore, we need to improve our lower bound.
\begin{propn}\label{clm15}
If $c_2 \geq  N+1$, the $2N$th Betti number of the moduli space $M_{\P,H}(2,-H, c_2)$ stabilize.
\end{propn}
\begin{proof}
Following the proof of Theorem \ref{theorem26}, it is enough to show that when $d > N+1$ ,the coefficient of $\L^{-N}q^d$ in $q^{\frac{1}{4}}(1 - \L^{-1})(1-q)G_{2,-H}(q)$ is zero.

Each term contributing to the coefficient of $\L^{-N}q^d$ in $q^{\frac{1}{4}}(1 - \L^{-1})(1-q)G_{2,-H}(q)$ arises from a pair of equations 
\begin{align*}
d &= \frac{1}{4} + \Delta' \\
-N &= \varepsilon - N'
\end{align*}
where $\varepsilon \in \left\lbrace -1,0 \right\rbrace$ accounts for the contribution to the coefficient coming from $(1 - \L^{-1})$, and $(\Delta', N'')$ accounts for the contribution coming from terms in coefficient of $\L^{-N'}q^{\Delta'}$ in $(1-q)G_{2,-H}(q)$. 

It follows from Theorem \ref{clm14} that for the co-efficient of $\L^{-N'}q^{\Delta'}$ to be nonzero, we must have $\Delta' \leq N' + C_0$. Consequently, we must have 
\begin{equation}
d - \frac{1}{4} = \Delta' \leq N' + C_0 = N + \varepsilon + C_0 \leq N + C_0
\end{equation}
As a result, for $d > N + \left\lfloor \frac{1}{4} + C_0 \right\rfloor$, the coefficient of $\L^{-N}q^d$ in $q^{\frac{1}{4}}(1 - \L^{-1}) (1-q)G_{2,-H}(q)$ must be zero. Therefore, to complete the proof of our Claim, we need to figure out the value of $C_0$.

It follows from the proof of Proposition \ref{boundMF} that to compute $C_0$, we need to compute 
\[  \frac{1}{2}\left(r^2 - \sum_{i=1}^l r_i^2 \right) - \frac{1}{2}\kappa \]
where $l=2$, $r=2$, $r_1 = r_2 = 1$, and $\kappa$ is a lower bound for
\[ 2\left( 2\Delta - \Delta_1 - \Delta_2 \right) + \left(c_2 - c_1 \right)\cdot K_{\F} \]
except for the case $l=2$ and $(c_2 - c_1)\cdot F = -1$.

Let $c_1 = a_1 E + b_1 F$ and $c_2 = a_2 E + b_2 F$. Since $c_1 + c_2 = -E-F$, we have $a_1 + a_2 = -1$ and $b_1 + b_2 = -1$. Moreover, we must have $a_2 - a_1 \neq -1$. Using Yoshioka's relation (equation \ref{eqnyos}) yields 
\[ \left( 2 \Delta - \Delta_1 - \Delta_2 \right) = - \frac{1}{4}\left( c_1 - c_2 \right)^2 
= \frac{1}{4}\left(2a_1 +1\right)^2 - \frac{1}{2}\left( 2a_1 +1\right) \left(2b_1 +1 \right) \]
Since $K_{\F} = -2E -3F$, we see that 
\[ \left( c_2 - c_1 \right) \cdot K_{\F} = \left( 2a_1 +1 \right) + 2 \left( 2b_1 +1 \right) \]
Therefore, we have 
\[ 2(2 \Delta - \Delta_1 - \Delta_2 ) + (c_2 - c_1)\cdot K_{\F} = 2 a_1^2 + 2 a_1 + 2b_1 - 4a_1 b_1 + \frac{5}{2} \]
Clearly $a_1^2 + a_1 \geq 0$ for all integer values of $a_1$. Thus, we need to find a lower  bound for $2b_1(1-2a_1)$.

Recall that as per the definition of $S^{\mu}(\{1,c_1\},\{1,c_2\},F,E+F)$ (see equation \ref{defnSmu}, \ref{eqnSmu}) we have two cases 
\begin{enumerate}[$\qquad \,$ A)]
\item $a_1 > -\frac{1}{2}$ and $b_1 \leq - \frac{1}{2}$
\item $a_1 \leq - \frac{1}{2}$ and $b_1 > - \frac{1}{2}$
\end{enumerate}
Since $a_1$ and $b_1$ are integers, in Case A, we see that $a_1 \geq 0$ and $-b_1 \geq 1$. When $a_1 =0$, we must have $a_2 = -1$, whence $a_2 - a_1 = -1$ which is not possible by our assumption. Hence, we must have $a_1 \geq 1$, which yields
\[ 2b_1(1 - 2a_1) = (2a_1 -1) (-2b_1) \geq \left( 2\left( 1\right) -1 \right) \left(2(1)\right) = 2 \]
Similarly, in Case B, we see that $-a_1 \geq 1$ and $b_1 \geq 0$, thereby yielding
\[ 2b_1(1 - 2a_1) \geq \left(2(0)\right) \left( 1 + 2(1)\right) = 0  \]
In either case we see that $2b_1(1 - 2a_1) \geq 0$, and hence we can take $\kappa = \frac{5}{2}$.

Clearly, in our case $r=2$ and $r_1 = r_2 = 1$, whence $\frac{1}{2}\left( r^2 - r_1^2 - r_2^2 \right) = 1$. Following the proof of Proposition \ref{boundMF}, we see that 
\[ C_0 = \max \left\lbrace 0, 1 +  1 - \frac{1}{2}\kappa, 1 - \frac{3}{4} +\left(\left\lceil \frac{-1}{2}\right\rceil - \frac{-1}{2}\right) \right\rbrace = 2 - \frac{5}{4} \]
In summary, for the coefficient of $\L^{-N}q^d$ to be nonzero, we must have 
\[d \leq N + \frac{1}{4} + 2 - \frac{5}{4} = N+1\]
In conclusion, when $d > N+1$, the coefficient of $\L^{-N}q^d$ in $q^{\frac{1}{4}}(1 - \L^{-1})(1-q)G_{2,-H}(q)$ is zero.
\end{proof}

Manschot \cite{man11}[Table 1], \cite{man}[Table 1] computed the Betti numbers of the moduli space $M_{\P,H}(3,-H,c_2)$ and the virtual Betti numbers of the moduli space $M_{\P,H}(4,2H, c_2)$. We observe from the tables in these papers that the Betti numbers of $M_{\P,H}(3,-H,c_2)$ stabilize when $c_2 \geq N + 2$ and the virtual Betti numbers of $M_{\P,H}(4,2H,c_2)$ stabilize when $c_2 \geq N+ 3$. In the first case, we have $r=3$ and $a = -1$, we get from Theorem \ref{theorem26} that the Betti numbers stabilize when $c_2 \geq N + 5$. 

As our second example, we scrutinize the Betti numbers of the moduli space $M_{\P,H}(4,H,c_2)$. In this case, Theorem \ref{theorem26} yields the stabilization of the Betti numbers when $c_2 \geq N + 8$. We improve this bound in the following Proposition. 
\begin{propn}\label{clm16}
If $c_2 \geq N + 5$, the $2N$-th Betti number of the moduli space $M_{\P,H}(4,H,c_2)$ stabilize.
\end{propn}
\begin{proof}
Following the proof of Theorem \ref{theorem26}, it is enough to show that when $d > N + 5$, the coefficient of $\L^{-N}q^d$ in $q^{\frac{3}{8}}(1 - \L^{-1})(1-q)G_{4,H}(q)$ is zero.

Each term contributing to the coefficient of $\L^{-N}q^{d}$ in $q^{\frac{3}{8}}(1 - \L^{-1})(1-q)G_{4,H}(q)$ arises from a pair of equations 
\begin{align*}
d &= \frac{3}{8} + \Delta' \\
-N &= \varepsilon - N'
\end{align*}
where $\varepsilon \in \lbrace -1,0 \rbrace$ accounts for the contribution to the coefficient coming from $(1 - \L^{-1})$, and $(\Delta',N')$ accounts for the contribution coming from the terms in coefficient of $\L^{-N'}q^{\Delta'}$ in $(1-q)G_{4,H}(q)$.

It follows from Theorem \ref{clm14} that if the co-efficient of $\L^{-N'}q^{\Delta'}$ is non-zero, then we must have $\Delta' \leq N' + C_0$, whence, $d \leq N + \left\lfloor \frac{3}{8} + C_0 \right\rfloor$. Consequently, for $d > N + \left\lfloor \frac{3}{8} + C_0 \right\rfloor$, the coefficient of $\L^{-N}q^{d}$ in $q^{\frac{3}{8}}(1 - \L^{-1})(1-q)G_{4,H}(q)$ must be zero. Therefore, to complete our proof, we need to determine the value of $C_0$.

Adopting the notation used in proof of Proposition \ref{boundMF} and Lemma \ref{lowerbound} in our situation, we get $r = 4$, $a = b = 1$. Recall that $C_0$ is the maximum of the terms $1 + \frac{1}{2}\left( r^2 - \sum_{i=1}^l r_i^2\right) - \frac{1}{2}\kappa$ except the case when $l=2$ and $\mu_F(\gamma_2 ) - \mu_F(\gamma_1) = -1$ and the terms $\frac{r_1 r_2}{2r} + \left( \left\lceil \frac{br_2}{r}\right\rceil - \frac{br_2}{r}\right)$ for $r_1 + r_2 = r$, where $r = \sum_{i=1}^l r_i$, $a = \sum_{i=1}^l r_i a_i$, $s_i = \sum_{j=i}^l b_j$, $b=s_1$, and  $ \kappa$ is lower bound for $S_1 + S_2$, where 
\[ S_1 = (r-1)\sum_{i=1}^l r_i a_i^2 - \frac{r-1}{r}a^2 + \sum_{i=1}^l a_i r_i \left( \sum_{j=i+1}^l r_j - \sum_{j=1}^{i-1} r_j \right) \]
and 
\[ S_1 = 2 \sum_{i=2}^l ( (r-1)(a_i - a_{i-1}) + r_i + r_{i-1}) \left( \frac{b}{r}\sum_{j=i}^l r_j - s_i \right) \]

When $l=2$ and $(r_1,r_2) = (3,1)$, we see that $S_1 \geq - \frac{3}{4}$ with equality occurring at $(a_1,a_2) = (0,1)$. At the point $(0,1)$ we get $S_2 \geq \frac{7}{2}$, and hence, $S_1 + S_2 \geq \frac{11}{4}$. Since there are no other points $(a_1,a_2)$ satisfying $3a_1 + a_2 = 1$ at which $S_1 < \frac{11}{4}$, we can take $\kappa = \frac{11}{4}$, and we get $1 + \frac{1}{2}\left( r^2 - r_1^2 - r_2^2 \right) - \frac{1}{2} \kappa = \frac{21}{8}$.

When $l=2$ and $(r_1,r_2) = (1,3)$, we see that $S_1 \geq \frac{21}{4}$ with equality occurring at $(a_1,a_2) = (1,0)$, and $S_2 \geq 1$. Thus, we can take $\kappa = \frac{25}{4}$, and we get $1 + \frac{1}{2}\left( r^2 - r_1^2 - r_2^2 \right) - \frac{1}{2} \kappa = \frac{7}{8}$.

When $l=2$ and $(r_1,r_2) = (2,2)$, there is no integer solution for $2 a_1 + 2a_2 = 1$. Thus, we ignore this case. 

When $l=3$ and $(r_1,r_2,r_3) = (2,1,1)$, we see that $S_1 \geq - \frac{3}{4}$ with equality occurring at $(a_1,a_2,a_3) = (0,0,1)$. At this point we get $S_2 \geq 4$, whence $S_1 + S_2 \geq \frac{5}{2}$. The only other point $(a_1,a_2,a_3)$ with $S_1 \leq \frac{5}{2}$ is $(0,1,0)$ at which $S_1 = \frac{5}{4}$ and $S_2 \geq \frac{9}{2}$, and thus $S_1 + S_2 \geq \frac{23}{4}$. Therefore, we can take $\kappa = \frac{5}{2}$, and we get $1 + \frac{1}{2}\left( r^2 - r_1^2 - r_2^2 - r_3^2 \right) - \frac{1}{2} \kappa = \frac{19}{4}$.

When $l=3$ and $(r_1,r_2,r_3) = (1,2,1)$, we see that $S_1 \geq - \frac{3}{4}$ with equality occurring at $(0,0,1)$. At this point, we see that $S_2 \geq 6$, whence $S_1 + S_2 \geq \frac{21}{4}$. At every other point $(a_1,a_2,a_3)$ with $a_1 + 2a_2 + a_3 = 1$, we have $S_1 \geq \frac{21}{4}$. As a consequence, we can take $\kappa = \frac{21}{4}$, and we get $1 + \frac{1}{2}\left( r^2 - r_1^2 - r_2^2 - r_3^2 \right) - \frac{1}{2} \kappa = \frac{27}{8}$.

When $l =3$ and $(r_1,r_2,r_3) = (1,1,2)$, we see that $S_1 \geq \frac{5}{4}$ with equality occurring at $(a_1,a_2,a_3) = (-1,0,1)$. At this point, we see that $S_2 \geq 6$, and thus $S_1 + S_2 \geq \frac{29}{4}$. The other points $(a_1,a_2,a_3)$ satisfying $a_1 + a_2 + 2a_3 = 1$ at which $S_1 \leq \frac{29}{4}$ are $(0,1,0), \, (0,-1,1),\, (1,0,0)$. Analyzing $S_1$ and $S_2$ at these points, we see that $S_1 + S_2 $ may attain the least possible value $\frac{25}{4}$. Thus, we take $\kappa = \frac{25}{4}$, and we see that $1 + \frac{1}{2}\left( r^2 - r_1^2 - r_2^2 - r_3^2 \right) - \frac{1}{2} \kappa = \frac{23}{8}$.

When $l=4$ and $(r_1,r_2,r_3,r_4) = (1,1,1,1)$, we see that $S_1 \geq - \frac{3}{4}$ with equality occurring at $(a_1,a_2,a_3,a_4) = (0,0,0,1)$. At this point, we see that $S_2 \geq 6$, and thus $S_1 + S_2 \geq \frac{21}{4}$. The other points $(a_1,a_2,a_3,a_4)$ with $a_1 + a_2 + a_3 + a_4 = 1$ at which $S_1 \leq \frac{21}{4}$ are $(-1,0,0,2)$, $(0,-1,1,1)$, $(-1,1,1,0)$, $(-1,1,0,1)$, $(-1,0,1,1)$, $(1,0,0,0)$, $(0,1,0,0)$, and $(0,0,1,0)$. However, we see that at each of these points we have $S_1 + S_1 \geq \frac{21}{4}$. Hence, we can take $\kappa = \frac{21}{4}$, and we get $1 + \frac{1}{2}\left( r^2 - r_1^2 - r_2^2 - r_3^2 - r_4^2 \right) - \frac{1}{2} \kappa =\frac{35}{8}$. 

Finally, since $b=1$, $r = 4$, and $1 \leq r_2 \leq 3$, we see that $\frac{(r-r_2)r_2}{2r} + 1 - \frac{r_2}{r}$ attains maximum value of $\frac{9}{8}$ at $r_2 = 1$.

In conclusion, we can take $C_0 = \frac{19}{4}$, and we get that when $d > N + 5$ the coefficient of $\L^{-N}q^d$ in $q^{\frac{3}{8}}(1 - \L^{-1})(1-q)G_{4,H}(q)$ is zero.
\end{proof}

In summary, as we see in our examples (Proposition \ref{clm15}, \ref{clm16}), the constant $C_0$ in Proposition \ref{boundMF} can be improved further, which will lead to better bounds for the stabilization of Betti numbers in Theorem \ref{theorem26}.
{\small
\renewcommand{\refname}{\textsc{References}}

{\footnotesize \textsc{Department of Mathematics, Statistics and CS, University of Illinois at Chicago, Chicago, IL, 60607}.\\
\textit{E-mail}: \texttt{smanda9@uic.edu}}}
\end{document}